\documentclass[10pt,reqno]{amsart}
\usepackage[T1]{fontenc}
\usepackage[utf8]{inputenc}
\usepackage{lmodern}
\usepackage{amssymb,mathtools}
\usepackage{tikz-cd}
\usepackage{aliascnt}
\usepackage{enumitem}
\usepackage[expansion=false]{microtype}
\usepackage{needspace,etoolbox}
\usepackage[hidelinks]{hyperref}
\usepackage[nameinlink,noabbrev]{cleveref}
\hypersetup{
pdftitle={Stationary common-neighborhood properties and partition hypotheses},
pdfauthor={Xiang Li},
pdfkeywords={highly connected Ramsey relations, partition hypotheses, weakly compact cardinals, Tukey order, Cech cohomology},
pdfsubject={Stationary common-neighborhood properties and partition hypotheses}}
\numberwithin{equation}{section}
\newtheorem{theorem}{Theorem}[section]
\newaliascnt{lemma}{theorem}
\newtheorem{lemma}[lemma]{Lemma}
\aliascntresetthe{lemma}
\crefname{lemma}{lemma}{lemmas}
\Crefname{lemma}{Lemma}{Lemmas}
\newaliascnt{corollary}{theorem}
\newtheorem{corollary}[corollary]{Corollary}
\aliascntresetthe{corollary}
\crefname{corollary}{corollary}{corollaries}
\Crefname{corollary}{Corollary}{Corollaries}
\newaliascnt{fact}{theorem}

\aliascntresetthe{fact}
\crefname{fact}{fact}{facts}
\Crefname{fact}{Fact}{Facts}
\newaliascnt{observation}{theorem}
\newtheorem{observation}[observation]{Observation}
\aliascntresetthe{observation}
\crefname{observation}{observation}{observations}
\Crefname{observation}{Observation}{Observations}
\newaliascnt{question}{theorem}
\newtheorem{question}[question]{Question}
\aliascntresetthe{question}
\crefname{question}{question}{questions}
\Crefname{question}{Question}{Questions}
\newaliascnt{proposition}{theorem}
\newtheorem{proposition}[proposition]{Proposition}
\aliascntresetthe{proposition}
\crefname{proposition}{proposition}{propositions}
\Crefname{proposition}{Proposition}{Propositions}

\theoremstyle{definition}
\newaliascnt{definition}{theorem}
\newtheorem{definition}[definition]{Definition}
\aliascntresetthe{definition}
\crefname{definition}{definition}{definitions}
\Crefname{definition}{Definition}{Definitions}
\newaliascnt{remark}{theorem}

\aliascntresetthe{remark}
\crefname{remark}{remark}{remarks}
\Crefname{remark}{Remark}{Remarks}
\newcommand{\keepstatement}[1]{\AtBeginEnvironment{#1}{\Needspace{5\baselineskip}}}
\forcsvlist{\keepstatement}{theorem,lemma,corollary,fact,observation,question,proposition,definition,remark,namedlemma,namedclaim}

\newcommand{\PP}{\mathbb P}
\newcommand{\QQ}{\mathbb Q}
\newcommand{\Coll}{\operatorname{Coll}}
\newcommand{\dom}{\operatorname{dom}}
\newcommand{\ran}{\operatorname{ran}}
\newcommand{\res}{\mathbin{\upharpoonright}}
\newcommand{\hc}{\mathrm{hc}}
\newcommand{\ZFC}{\mathsf{ZFC}}
\newcommand{\Z}{\mathbb Z}

\newcommand{\Pow}{\mathcal P}
\newcommand{\RO}{\operatorname{RO}}
\newcommand{\bbB}{\mathbb B}
\newcommand{\mcD}{\mathcal D}
\newcommand{\hcarr}{\to_{\hc,<3}}
\newcommand{\FCN}{\mathsf{FCN}}
\newcommand{\SFCN}{\mathsf{SFCN}}
\newcommand{\PH}{\operatorname{PH}}
\newcommand{\calA}{\mathcal A}
\newcommand{\Freq}{\operatorname{Freq}}
\newcommand{\concat}{\mathbin{{}^\frown}}

\title[Stationary common-neighborhood properties]
{Stationary common-neighborhood properties and partition hypotheses}
\author{Xiang Li}
\address{Department of Mathematics, University of California, Los Angeles,
Los Angeles, California, USA}
\subjclass[2020]{Primary 03E02. Secondary 03E04, 03E35, 03E55, 55N05}
\keywords{Highly connected Ramsey relations, partition hypotheses,
weakly compact cardinals, Tukey order, \v{C}ech cohomology}
\date{}
\setlist[enumerate]{label=\textup{(\arabic*)},leftmargin=*,itemsep=2pt}

\begin{document}
\begin{abstract}
We use stationary common-neighborhood properties to study highly
connected Ramsey relations and partition hypotheses. For weakly compact $\kappa$,
$\Coll(\omega_1,{<}\kappa)$ forces
$\omega_2\to_{\hc,<5}(\omega_2)^2_\omega$
and $\PH_1(\omega_2)$.
If $\kappa$ is $T^{\kappa^+}_{\omega_1}$-Ramsey, the same collapse
forces that every countable coloring of $[\omega_2]^2$ has a
stationary set $X\subseteq\omega_2$ and a color $i$ such that
every finite subset of $X$ has stationarily many color-$i$
common neighbors in $X$.
From one weakly compact cardinal, we obtain a model of the
${<}5$-edge relation at $\omega_3$ and $\PH_1(\omega_3)$,
in which $\check H^2(\omega_3,A_d)\ne0$ for every
nontrivial abelian group $A$. This separates $\PH_1(\omega_3)$
from $\PH_2(\omega_3)$, with the exact consistency strength of
one weakly compact cardinal.
We also show that $\PH_1(\omega_2\times\omega_5)$ is equiconsistent
with two weakly compact cardinals.
\end{abstract}
\maketitle

\section{The results}\label{sec:introduction}

Throughout, $\kappa$ is a regular uncountable cardinal in the ground model $V$. For any poset $(P,\leq_P)$, we also write $P$ for
$(P,\leq_P)$.

A graph of size $\kappa$ is \emph{highly connected} if it
remains connected after deleting fewer than $\kappa$ vertices. The
relation $\kappa\to_{\hc}(\kappa)^2_\omega$ says that every countable
coloring of $[\kappa]^2$ has a monochromatic highly connected subgraph
of size $\kappa$~\cite[Definition~3]{BHS}. For an integer $m\geq2$, the
${<}m$-edge relation $\kappa\to_{\hc,<m}(\kappa)^2_\omega$ requires
the remaining graph to have diameter less than $m$
\cite[Definition~1.3]{HSZ}.

Write $NS_\kappa$ for the ideal of nonstationary subsets of $\kappa$.
Let $\mathcal J$ be a proper ideal on $\kappa$ containing
$[\kappa]^{<\kappa}$. For $X\in[\kappa]^\kappa$, let
$e_X:\kappa\to X$ be the increasing enumeration and define
$\mathcal J_X=\{e_X[A]:A\in\mathcal J\}$.
The ideal-indexed relation
$\kappa\to_{\mathcal J\text{-hc}}(\kappa)^2_\lambda$
says that every $c:[\kappa]^2\to\lambda$ has such an $X$ and
a color $i<\lambda$ such that the color-$i$ graph on $X$ remains
connected after every deletion from $\mathcal J_X$.
Its $<m$ refinement requires diameter less than $m$ after every
such deletion~\cite[Problem~8.6]{HSZ}.

The partition hypotheses $\PH_n(P,\lambda)$ were introduced by
Bannister--Bergfalk--Moore--Todorcevic in their study of higher derived
limits~\cite[Definitions~3.1--3.2 and Section~8]{BBMT}.
We give the definitions in Section~\ref{sec:partition-hypotheses}.
We write $\PH_n(P)$ for the countable-color case.

For a coloring $c:[\kappa]^2\to\lambda$ and $i<\lambda$, let
$N_i(x)=\{y\in\kappa\setminus\{x\}:c(x,y)=i\}$.
For $X\subseteq\kappa$ and finite $F\subseteq\kappa$, define the
\emph{color-$i$ common neighborhood} of $F$ in $X$ by
\[
 \Gamma_i^X(F)=\{y\in X:(\forall x\in F)\ y\in N_i(x)\}.
\]

Our main conclusion is the following property of $B\subseteq\kappa$
and $i<\lambda$:
\begin{equation}\label{eq:stationary-hubs}
 \forall F\in[B]^{<\kappa}\quad
 \{z\in B:
       (\forall x\in F)\ \Gamma_i^B(\{x,z\})\not\in NS_\kappa\}
       \not\in NS_\kappa.
\end{equation}

Suppose $\lambda<\kappa$ is infinite and every $\lambda$-coloring
has such $B,i$. Then, Proposition~\ref{prop:hub-deletions} and
Corollary~\ref{cor:hubs-ph} give
\[
 \kappa\to_{NS_\kappa\text{-hc},<5}(\kappa)^2_\lambda
 \qquad\text{and}\qquad \PH_1(\kappa,\lambda).
\]

\begin{theorem}\label{thm:main}
Suppose $\kappa$ is weakly compact and $\nu<\kappa$ is
uncountable regular. In the $\Coll(\nu,{<}\kappa)$-extension,
$\kappa=\nu^+$. For every infinite $\lambda<\nu$, each coloring
$c:[\kappa]^2\to\lambda$ has $B\subseteq\kappa$ and $i<\lambda$
satisfying \eqref{eq:stationary-hubs}.
\end{theorem}

For $\nu=\omega_1$, Theorem~\ref{thm:main} answers
\cite[Question~8.1]{HSZ} and \cite[Question~9.4]{BBMT}.
The earlier upper bound for $\PH_1(\omega_2)$ used a huge cardinal
\cite[Theorem~8.13 and Remark~8.14]{BBMT}.
The weakly compact bound is sharp for both principles.
For regular $\kappa>\omega_1$, each of
$\kappa\to_{\hc}(\kappa)^2_\omega$ and $\PH_1(\kappa)$ implies
$\neg\square(\kappa)$
\cite[Theorem~3.8]{LH}\cite[Proposition~8.19]{BBMT}.
Hence, $\kappa$ is weakly compact in $L$.
Therefore, the $NS_{\omega_2}$ instance of \cite[Problem~8.6]{HSZ}
is equiconsistent with a weakly compact cardinal.

The \emph{finite common-neighborhood property} is
\begin{equation}\label{eq:fcn}
 \forall F\in[X]^{<\omega}\quad |\Gamma_i^X(F)|=\kappa.
\end{equation}
Its \emph{stationary} version is
\begin{equation}\label{eq:stationary-fcn}
 \forall F\in[X]^{<\omega}\quad
       \Gamma_i^X(F)\text{ is stationary in }\kappa.
\end{equation}
Write $\FCN_\omega(\kappa)$ and $\SFCN_\omega(\kappa)$ when every
countable coloring has $X,i$ satisfying the respective property.
Use subscripts $\lambda$ for the same properties with $\lambda$
colors. Choosing a common neighbor outside the deleted set gives
\[
 \SFCN_\omega(\kappa)\ \Longrightarrow\ \FCN_\omega(\kappa)
 \ \Longrightarrow\ \kappa\hcarr(\kappa)^2_\omega.
\]

Hru\v{s}\'ak--Shelah--Zhang obtained the ${<}4$-edge relation from a
countably complete uniform ideal with a countably closed dense family
of positive sets~\cite[Theorem~5.1]{HSZ}. They asked whether the
${<}3$-edge relation at $\omega_2$ is consistent~\cite[Question~8.2]{HSZ}.
We obtain the following stationary form using the property in
Definition~\ref{def:height-ramsey}(2).

\begin{theorem}\label{thm:ramseycollapse}
Suppose $\kappa$ is $T^{\kappa^+}_{\omega_1}$-Ramsey and
$\nu<\kappa$ is uncountable regular. Then, $\Coll(\nu,{<}\kappa)$
forces $\kappa=\nu^+$ and $\SFCN_\omega(\kappa)$.
\end{theorem}

This hypothesis has strictly lower consistency strength than
Holy--Schlicht's $\omega_1$-Ramsey property
\cite[Lemma~9.14]{HolyLucke}.
By Proposition~\ref{prop:positive-ideal-connectivity}, the conclusion
also gives $\kappa\to_{NS_\kappa\text{-hc},<3}(\kappa)^2_\omega$.

Theorem~\ref{thm:tukey} shows that $\PH_n(-,\lambda)$ is downward
closed in the Tukey order for every $n<\omega$ and cardinal $\lambda$.
This answers \cite[Question~9.11]{BBMT}.
Fix integers $2\leq a$ and $b\geq a+3$.
From two weakly compact cardinals, we obtain a model in which
$\PH_1(P)$ holds for every nonempty directed quasi-order
$P\leq_T\omega_a\times\omega_b$
(Corollary~\ref{cor:product-consistency}).
For $a=2$ and $b=5$, the consistency strength is exactly two
weakly compact cardinals (Corollary~\ref{cor:product-equiconsistency}).

\begin{theorem}\label{thm:ph-separation}
Relative to a weakly compact cardinal, there is a model in which
every countable coloring of $[\omega_3]^2$ has $B\subseteq\omega_3$
and a color satisfying \eqref{eq:stationary-hubs}.
In the same model,
\[
 \check H^2(\omega_3,A_d)\ne0
 \qquad\text{for every nontrivial abelian group }A.
\]
\end{theorem}

Here $A_d$ is the sheaf of locally constant $A$-valued functions,
and $\check H^n$ denotes \v{C}ech cohomology.
In Theorem~\ref{thm:ph-trivializes}, we prove
\[
 \PH_n(\kappa)\ \Longrightarrow\ \check H^n(\kappa,A_d)=0
 \qquad(1\leq n<\omega,\ A\text{ abelian}).
\]
The model of Theorem~\ref{thm:ph-separation} satisfies
$\check H^1(\omega_3,A_d)=0$ for every abelian group $A$
since \eqref{eq:stationary-hubs} gives $\PH_1(\omega_3)$.
The nonvanishing of $\check H^2$ gives $\neg\PH_2(\omega_3)$.
The lower bound above shows that Theorem~\ref{thm:ph-separation}
is equiconsistent with a weakly compact cardinal.
Proposition~\ref{prop:thread-products} also answers
\cite[Question~4.2]{BLH} negatively for two-coherent families.
This result requires no large-cardinal assumptions.

\section{Stationary patterns and four-edge paths}\label{sec:patterns}
Fix a coloring $c:[\kappa]^2\to\lambda$, where $\lambda<\kappa$.

\subsection{Connectivity after deletions}\label{sec:ideal-connectivity}

\begin{proposition}\label{prop:positive-ideal-connectivity}
Let $I\supseteq NS_\kappa$ be an ideal, and let $i<\lambda$ and
$X\subseteq\kappa$ satisfy $\Gamma_i^X(F)\in I^+$ for every finite
$F\subseteq X$.
Then, $X$ witnesses $\kappa\to_{I\text{-hc},<3}(\kappa)^2_\lambda$
for $c$ in color $i$.
\end{proposition}
\begin{proof}
Let $Z\in I$. For finite $F\subseteq X\setminus Z$,
$\Gamma_i^{X\setminus Z}(F)=\Gamma_i^X(F)\setminus Z\in I^+$.
Taking $F=\varnothing$ gives $X\setminus Z\in I^+$, and taking
$F=\{x,y\}$ gives the diameter bound.

Since $NS_\kappa\subseteq I$, we have $|X|=\kappa$. Let $e=e_X$
and take the club $C_e=\{\delta<\kappa:e``\delta\subseteq\delta\}$.
If $e(\alpha)\in C_e$, then $e(\alpha)=\alpha$, for otherwise
$e(\alpha) \in e``(e(\alpha))\subseteq e (\alpha)$. Hence, for every
$A\subseteq\kappa$,
$e[A]\subseteq A\cup(\kappa\setminus C_e)$.
Thus, $A\in I$ implies $e_X[A]\in I$, and the
diameter bound above applies with $Z=e_X[A]$.
\end{proof}

Taking $I=NS_\kappa$ gives
$\SFCN_\omega(\kappa)\ \Longrightarrow\ \kappa\to_{NS_\kappa\text{-hc},<3}(\kappa)^2_\omega$.

\begin{proposition}\label{prop:hub-deletions}
Suppose $B,i$ satisfy \eqref{eq:stationary-hubs}.
Then, $B$ witnesses $\kappa\to_{NS_\kappa\text{-hc},<5}(\kappa)^2_\lambda$
for $c$ in color $i$.
\end{proposition}
\begin{proof}
Let $Z$ be nonstationary and let $x\ne y$ be in $B\setminus Z$.
By \eqref{eq:stationary-hubs}, choose $z\in B\setminus(Z\cup\{x,y\})$
such that $\Gamma_i^B(\{x,z\})$ and $\Gamma_i^B(\{y,z\})$ are stationary.
Choose distinct common neighbors $u$ of $x,z$ and $v$ of $y,z$
outside $Z\cup\{x,y\}$. Then, $x-u-z-v-y$ is a color-$i$ path.
The enumeration argument in Proposition~\ref{prop:positive-ideal-connectivity}
gives the ideal-indexed relation.
\end{proof}

\subsection{Internal common neighbors}

For the rest of this section, assume that $\lambda$ is infinite and
$\theta^{<\lambda}<\kappa$ for every infinite $\theta<\kappa$.
This is automatic when $\lambda=\omega$.
All sets and cardinalities are computed in the ground model unless
an extension is specified. Let $\mathcal C$ and $\mathcal S$ be its
families of club and stationary subsets of $\kappa$.

For a partial function $t:\kappa\to\lambda$ of size at most $\lambda$, let
\[
 R_c(t)=\bigcap_{\alpha\in\dom(t)}N_{t(\alpha)}(\alpha).
\]
If $t\subseteq s$, then $R_c(s)\subseteq R_c(t)$.

\begin{lemma}\label{lem:extraction}
Suppose $\QQ$ is ${<}\lambda^+$-closed and $\dot f$ is a
$\QQ$-name such that
\begin{equation}\label{eq:pattern-hypothesis}
 1_{\QQ}\Vdash_{\QQ}
 \dot f:\kappa\to\lambda\ \wedge\
 (\forall t\in[\dot f]^{\leq\lambda})\
                         R_c(t)\in\check{\mathcal S}.
\end{equation}
Then, there are $p\in\QQ$, a ground-model stationary set $A$, and names
$\langle\dot X_i:i<\lambda\rangle$ for subsets of $\kappa$ such that
$p\Vdash_{\QQ} A\subseteq\bigcup_{i<\lambda}\dot X_i$ and
\begin{equation}\label{eq:internal-neighbors}
 p\Vdash_{\QQ}(\forall i<\lambda)(\forall F\in[\dot X_i]^2)
       (\forall C\in\check{\mathcal C})\
                   \Gamma_i^{\dot X_i}(F)\cap C\ne\varnothing.
\end{equation}
\end{lemma}
\begin{proof}
In a $\QQ$-generic extension, write $f=\dot f^H$ and
$Y_i=f^{-1}\{i\}$. Call $i$ good if $R_c(t)\cap Y_i$ meets
every ground club for every $t\in[f]^{\leq\lambda}$.
Let $X_i=Y_i$ for good $i$, and $X_i=\varnothing$ otherwise.
For $x\ne y$ in $X_i$, goodness applied to
$t=\{(x,i),(y,i)\}$ gives \eqref{eq:internal-neighbors}.

For each bad $i$, choose $t_i\in[f]^{\leq\lambda}$ and
$C_i\in\mathcal C$ with $R_c(t_i)\cap Y_i\cap C_i=\varnothing$.
For good $i$, use $t_i=\varnothing$ and $C_i=\kappa$.
The forcing $\QQ$ adds no sequences of ordinals of length at most
$\lambda$. Hence, each $t_i$ and both sequences
$\langle t_i:i<\lambda\rangle$, $\langle C_i:i<\lambda\rangle$
are in the ground model. Thus,
$s=\bigcup_{i<\lambda}t_i$ and $C=\bigcap_{i<\lambda}C_i$ are
ground-model sets, and $C$ is club since $\lambda<\kappa$.
By \eqref{eq:pattern-hypothesis}, $A=R_c(s)\cap C$ is stationary
in the ground model. For bad $i$,
$A\cap Y_i\subseteq R_c(t_i)\cap C_i\cap Y_i=\varnothing$.
Hence, $A\subseteq\bigcup_iX_i$. Choose a ground-model sequence
of names for the $X_i$ and apply the truth lemma.
\end{proof}

\subsection{Descending a cover}\label{sec:descent}

\begin{lemma}\label{lem:overlap}
Let $\QQ$ be ${<}\lambda^+$-closed, let $A$ be stationary, and
suppose $1_{\QQ}\Vdash_{\QQ} A\subseteq\bigcup_{i<\lambda}\dot X_i$.
For $p\in\QQ$, let
\[
 U_i(p)=\{\alpha<\kappa:
              (\exists q\leq p)\ q\Vdash_{\QQ}\alpha\in\dot X_i\}.
\]
There are $p\in\QQ$ and $i<\lambda$ such that
\begin{equation}\label{eq:overlap}
 \bigcap_{q\in D}U_i(q)\text{ is stationary}
 \qquad(D\subseteq\QQ\mathbin{\downarrow}p,\ |D|<\kappa).
\end{equation}
\end{lemma}
\begin{proof}
Suppose otherwise. Build a tree of conditions
$\langle p_s:s\in\mathcal T\rangle$ of height $\lambda$, starting
with $p_\varnothing=1_{\QQ}$. At a node $s$ of height $i<\lambda$,
choose $0<\tau_s<\kappa$ and extensions
$p_{s^\frown\xi}\leq p_s$ for $\xi<\tau_s$ such that
\[
 E_s=\bigcap_{\xi<\tau_s}U_i(p_{s^\frown\xi})
       \text{ is nonstationary}.
\]
At a nonzero limit level $\delta<\lambda$, include every branch
through the earlier levels and assign it a lower bound, using
${<}\lambda^+$-closure.

Every level has size less than $\kappa$, by induction.
At successor levels, use regularity of $\kappa$.
At a limit $\delta<\lambda$, choose an infinite cardinal $\theta<\kappa$
bounding the number of nodes on earlier levels.
There are at most $\theta^{|\delta|}\leq\theta^{<\lambda}<\kappa$ branches.
The tree has $\lambda<\kappa$ levels, so it has fewer than $\kappa$ nodes.

Choose $\alpha\in A\setminus\bigcup_{s\in\mathcal T}E_s$.
Follow a branch, at each level $i$ choosing a successor $s^\frown\xi$
with $\alpha\notin U_i(p_{s^\frown\xi})$.
At limit levels use the lower bounds already chosen.
Then, a lower bound $q$ of this branch of length $\lambda$ forces
$\alpha\notin\bigcup_{i<\lambda}\dot X_i$, contrary to the cover of $A$.
\end{proof}

\begin{lemma}\label{lem:transfer}
If $\QQ$ is ${<}\lambda^+$-closed and $\dot f$ satisfies
\eqref{eq:pattern-hypothesis}, then there are $B\subseteq\kappa$
and $i<\lambda$ satisfying \eqref{eq:stationary-hubs}
in the ground model.
\end{lemma}
\begin{proof}
Obtain $p_0,A,\langle\dot X_i:i<\lambda\rangle$ from
Lemma~\ref{lem:extraction}. Apply Lemma~\ref{lem:overlap} below $p_0$. Let $p,i$ witness \eqref{eq:overlap} and set $B=U_i(p)$.
For $F\in[B]^{<\kappa}$, choose $q_x\leq p$ forcing
$x\in\dot X_i$ for each $x\in F$.
Apply \eqref{eq:overlap} to $\{p\}\cup\{q_x:x\in F\}$.
There are stationarily many
$z\in\bigl(B\cap\bigcap_{x\in F}U_i(q_x)\bigr)\setminus F$.
For such a $z$ and $x\in F$, choose $r\leq q_x$ forcing
$x,z\in\dot X_i$. Given a ground club $C$,
\eqref{eq:internal-neighbors} gives $u\in C$ and an extension of $r$
forcing $u\in\dot X_i$ and $c(x,u)=c(z,u)=i$.
Then, $u\in U_i(p)=B$, and the edge equalities hold in the ground
model. Hence, $\Gamma_i^B(\{x,z\})$ is stationary there.
\end{proof}

\section{The collapse and lifted embeddings}\label{sec:collapse}

Let $\kappa$ be strongly inaccessible and let $\nu<\kappa$ be
uncountable regular. We use the following presentation of
$\PP=\Coll(\nu,{<}\kappa)$. A condition is a partial function $p$
on $[\nu,\kappa)\times\nu$ of size less than $\nu$, with
$p(\alpha,\xi)<\alpha$ for $(\alpha,\xi)\in\dom(p)$.
The order is reverse inclusion.

The forcing $\PP$ is ${<}\nu$-closed and $\kappa$-c.c.
It has size $\kappa$ and forces $\kappa=\nu^+$.
Every condition is in $V_\kappa$.
The poset $\PP$ is in $H(\kappa^+)$.

We use $\ZFC^-$ for $\ZFC$ without Power Set, with Collection and the
well-ordering form of Choice.

\begin{lemma}\label{lem:small-names}
Let $\kappa$ be an infinite cardinal and let $P\in H(\kappa^+)$ be a forcing.
\begin{enumerate}
\item Let $S\in H(\kappa^+)$ and let $\tau$ be a $P$-name.
There is a $P$-name $\sigma\in H(\kappa^+)$ such that
$1_P\Vdash_P\sigma=\tau\cap\check S$.
For fixed $S$, these names $\sigma$ can be chosen from a family
of size at most $2^\kappa$.
\item Suppose $P$ preserves $\kappa$. Then,
\[
 (H(\kappa^+)^V)[G]=H(\kappa^+)^{V[G]}
\]
for every $V$-generic $G\subseteq P$.
If a $P$-name $\tau$ satisfies $1_P\Vdash_P\tau\in H(\check\kappa^+)$,
then there is a $P$-name $\sigma\in H(\kappa^+)^V$ such that
$
 1_P\Vdash_P\tau=\sigma.
$
\end{enumerate}
\end{lemma}
\begin{proof}
For (1), define
\begin{equation}\label{eq:small-name}
 \sigma=\{(\check s,p):s\in S,\ p\in P,
                         \ p\Vdash_P\check s\in\tau\}.
\end{equation}
The forcing theorem gives $1_P\Vdash_P\sigma=\tau\cap\check S$.
Since $P,S\in H(\kappa^+)$, we have $\sigma\in H(\kappa^+)$.
Each name of the form \eqref{eq:small-name} is a subset of
$\{\check s:s\in S\}\times P$, so there are at most $2^\kappa$ such names.

For (2), let $T=H(\kappa^+)^V$. Since $|P|\leq\kappa$, forcing
with $P$ preserves $\kappa^+$, and $T[G]\subseteq H(\kappa^+)^{V[G]}$.
Given $x\in H(\kappa^+)^{V[G]}$, code $\operatorname{tc}(\{x\})$
by a well-founded extensional relation $E$ on an ordinal $\delta\leq\kappa$.
By (1), $E$ has a name in $T$, so $E\in T[G]$.
Since $T[G]\models\ZFC^-$, the Mostowski collapse of $(\delta,E)$
is in $T[G]$ and agrees with the ambient collapse. Thus, $x\in T[G]$.

Now suppose $\tau$ is forced to be in $H(\kappa^+)$.
By the equality above, every $p\in P$
has an extension $q$ such that $q\Vdash\tau=\sigma$ for some
$P$-name $\sigma\in T$. Choose a maximal antichain
$\mathcal A\subseteq P$ and names $\sigma_p\in T$ such that
$p\Vdash\tau=\sigma_p$ for each $p\in\mathcal A$.
By the mixing lemma, there is a $P$-name $\sigma$ such that
$p\Vdash\sigma=\sigma_p$ for every $p\in\mathcal A$.
The usual construction uses only subnames of the $\sigma_p$
and conditions of $P$. Since $|\mathcal A|\leq\kappa$ and
$P,\sigma_p\in H(\kappa^+)^V$, $\sigma\in T$.
Maximality of $\mathcal A$ gives $1_P\Vdash\tau=\sigma$.
\end{proof}

\subsection{Models and ultrapowers}\label{sec:models}

A \emph{weak $\kappa$-model} is a set $M$ of cardinality $\kappa$
with $\kappa+1\subseteq M$ and $(M,\in)\models\ZFC^-$.
Every weak $\kappa$-model $M\prec H(\kappa^+)$ is transitive.
Indeed, for nonempty $a\in M$, elementarity gives in $M$
a surjection from $\kappa$ onto $a$. All its values are in $M$
because $\kappa\subseteq M$, so $a\subseteq M$.
The sets $V_\kappa$ and $\PP$ are definable in $H(\kappa^+)$ from
$\kappa$ and $\nu$. Hence, $V_\kappa,\PP\in M$ by elementarity,
and $V_\kappa\subseteq M$ by transitivity.

For transitive $M$, an \emph{$M$-ultrafilter} on $\kappa$ is a
nonprincipal ultrafilter on $\Pow(\kappa)\cap M$ which is
$\kappa$-complete and normal for sequences in
$M$~\cite[Definition~2.1]{Gitman}.
Let $j:M\to N$ be the well-founded ultrapower by such a $U$,
with $N$ transitive. Normality gives $[\mathrm{id}]_U=\kappa$ and
\begin{equation}\label{eq:normal-form}
 N=\{j(h)(\kappa):h\in M\text{ is a function with domain }\kappa\}.
\end{equation}

An $M$-ultrafilter $U$ is \emph{weakly amenable} if $U\cap a\in M$
whenever $a\in M$
has cardinality $\kappa$ in $M$
\cite[Definition~2.4]{Gitman}.
An elementary embedding $j:M\to N$ between transitive models is
\emph{$\kappa$-powerset preserving} if its critical point is $\kappa$
and
$\Pow(\kappa)\cap M=\Pow(\kappa)\cap N$.
For the well-founded ultrapower by an $M$-ultrafilter $U$, this
is equivalent to weak amenability of $U$ \cite[Proposition~2.6]{Gitman}.

\begin{definition}\label{def:height-ramsey}
Let $\kappa$ be inaccessible.
\begin{enumerate}
\item $\kappa$ is \emph{weakly compact} if, for every
$x\in H(\kappa^+)$, there is a weak $\kappa$-model
$M\prec H(\kappa^+)$ with $x\in M$ and $M^\omega\subseteq M$,
carrying an $M$-ultrafilter $U$ on $\kappa$.
\item $\kappa$ is \emph{$T^{\kappa^+}_{\omega_1}$-Ramsey} if the
same condition holds with $U$ weakly amenable.
\end{enumerate}
\end{definition}

In both cases, countable closure of $M$ makes $U$ externally
countably complete, so its ultrapower is well founded.
The characterization in (1) follows from
\cite[Theorem~1.1(4),(6) and Proposition~2.3]{Gitman}.
In (1), $M$ may also be chosen with ${}^{<\kappa}M\subseteq M$.
Part (2) is \cite[Definition~9.4]{HolyLucke} with
$\alpha=\omega_1$, $\theta=\kappa^+$ and the trivial property $T$.
Their $\kappa$-amenability is weak amenability in our terminology.

Every $T^{\kappa^+}_{\omega_1}$-Ramsey cardinal is Ramsey
\cite[Proposition~2.8(3)]{Gitman}.
Holy--Schlicht's $\omega_1$-Ramsey property requires such elementary
models at arbitrarily large regular heights
\cite[Definition~5.1 and Theorem~5.5(c)]{HolySchlicht}.
Every $\omega_1$-Ramsey cardinal is a stationary limit of cardinals
$\gamma$ which are $T^{\gamma^+}_{\omega_1}$-Ramsey
\cite[Lemma~9.14 and the paragraph following Definition~9.12]{HolyLucke}.

\subsection{The collapse setup}\label{sec:collapse-setup}

Let $M$ be a transitive weak $\kappa$-model and let $j:M\to N$
be a well-founded ultrapower as in~\eqref{eq:normal-form}.
Suppose $\mu\leq\kappa$ is regular uncountable and
${}^{<\mu}M\subseteq M$ in $V$. Then, ${}^{<\mu}N\subseteq N$ in $V$ by
\cite[Proposition~2.3(5)]{Gitman}.
Let $P\in M$ be a forcing that adds no sequences of length less than
$\mu$ from ground-model sets, and let $G\subseteq P$ be $V$-generic.
Then, $M$ and $N$ remain closed under such sequences in $V[G]$.
Therefore, if $Q\in N$ and $N$ regards $Q$ as ${<}\mu$-closed,
then $Q$ is externally ${<}\mu$-closed in $V[G]$.

In $V[G]$, apply the ground closure criterion
\cite[Lemma~3.7]{GitmanJohnstone} to $M$ and $G$.
It gives ${}^{<\mu}(M[G])\subseteq M[G]$.

We use the \emph{lifting criterion} \cite[Lemma~3.1]{GitmanJohnstone}.
Let $j:M\to N$ be elementary between transitive models of $\ZFC^-$,
and let $P\in M$. Suppose $G\subseteq P$ is $M$-generic and
$K\subseteq j(P)$ is $N$-generic.
If $j``G\subseteq K$, then $j$ lifts to an elementary embedding
\[
 j^*:M[G]\longrightarrow N[K],\qquad
 j^*(\tau^G)=j(\tau)^K.
\]

\begin{lemma}\label{lem:factor-lift}
Let $\PP=\Coll(\nu,{<}\kappa)$ as above, and let $M,N$ be transitive
weak $\kappa$-models. Suppose $j:M\to N$ is the well-founded
ultrapower by an $M$-ultrafilter on $\kappa$, with
$V_\kappa\subseteq M$ and $\PP\in M$.
Define in $N$
$\QQ=\{r\in j(\PP):\dom(r)\subseteq[\kappa,j(\kappa))\times\nu\}$.
Then, $\PP\in N$.
The forcing $\PP$ is in $H(\kappa^+)$ in both models.
The forcing $\QQ$ is ${<}\nu$-closed in $N$.
Restriction to the two sets of coordinates gives an isomorphism in $N$,
\begin{equation}\label{eq:factorization}
 \begin{split}
 \pi:j(\PP)&\longrightarrow\PP\times\QQ,\\
 r&\longmapsto
 \bigl(r\res([\nu,\kappa)\times\nu),
       r\res([\kappa,j(\kappa))\times\nu)\bigr),
 \end{split}
\end{equation}
with $\pi(j(p))=(p,\varnothing)$ for every $p\in\PP$.
Let $G\subseteq\PP$ be generic over $V$. For every
$H\subseteq\QQ$ generic over $V[G]$, the filter
$K=\pi^{-1}[G\times H]$ is $N$-generic and $j$ lifts to
\begin{equation}\label{eq:lift}
 j_H:M[G]\longrightarrow N[K],\qquad
 j_H(\tau^G)=j(\tau)^K.
\end{equation}
\end{lemma}
\begin{proof}
The embedding $j$ fixes $V_\kappa$ pointwise, so
$V_\kappa\subseteq M\cap N$. Every condition of $\PP$ has an
enumeration of length less than $\nu$ in $V_\kappa$.
Thus, $\PP=\Coll(\nu,{<}\kappa)^M$.
By elementarity, $j(\PP)=\Coll(\nu,{<}j(\kappa))^N$, and
\[
 \PP=\{r\in j(\PP):\dom(r)\subseteq[\nu,\kappa)\times\nu\}\in N.
\]
Indeed, $j(p)=p$ for $p\in\PP$. Conversely, every $r$ on the
right satisfies the coordinate and value requirements for $\PP$,
and an enumeration in $N$ of length less than $\nu$ also witnesses
$|r|<\nu$ in $V$. Since $V_\kappa\subseteq M\cap N$ and $\kappa$
is strongly inaccessible in $V$, both models regard $\kappa$ as
strongly inaccessible. The usual size calculation shows that $\PP$
is in $H(\kappa^+)$ in both models.

Restriction to the two sets of coordinates gives $\pi$, with
inverse $(p,q)\mapsto p\cup q$. Since $j(p)=p$ for $p\in\PP$,
we have $\pi(j(p))=(p,\varnothing)$. In $N$, the union of a
decreasing sequence in $\QQ$ of length less than $\nu$ is a condition,
since $\nu$ is regular. Thus, $\QQ$ is ${<}\nu$-closed in $N$.

The filter $G\times H$ is $V$-generic for $\PP\times\QQ$.
Thus, $K$ is $N$-generic and contains $j``G$.
The lifting criterion gives~\eqref{eq:lift}.
\end{proof}

\subsection{Lifted names and stationary patterns}\label{sec:seed}

Fix an isomorphism $\pi:R\to P\times Q$. For a $P$-generic filter
$G$ and an $R$-name $\sigma$, define its partial evaluation recursively by
\[
 \sigma_G=\{(\tau_G,q):q\in Q,\ \exists p\in G\
                 ((\tau,\pi^{-1}((p,q)))\in\sigma)\}.
\]
Thus, $\sigma_G$ is a $Q$-name.
For a $Q$-generic $H$ over the $P$-extension, let
$K=\pi^{-1}[G\times H]$. Since $r\in K$ exactly when
$\pi(r)\in G\times H$, induction on name rank gives
$(\sigma_G)^H=\sigma^K$.
If $R,P,Q,\pi,\sigma\in N$ and $G$ is generic over $N$, the recursion
is carried out in $N[G]$, so $\sigma_G\in N[G]$.

\begin{lemma}\label{lem:ambient-stationarity}
Suppose $P\in H(\kappa^+)$ preserves $\kappa$ as a cardinal.
Let $M\prec H(\kappa^+)$ be a weak
$\kappa$-model containing $P$, and let $G\subseteq P$ be generic over
$V$. Then,
$M[G]\prec H(\kappa^+)^{V[G]}$.
If an ideal $I$ on $\Pow(\kappa)\cap M[G]$ contains the complements
of all clubs in $\kappa$ that are in $M[G]$, then every $I$-positive member of $\Pow(\kappa)\cap M[G]$
is stationary in $V[G]$.
\end{lemma}
\begin{proof}
Let $T=H(\kappa^+)^V$. By Lemma~\ref{lem:small-names}(2),
$T[G]=H(\kappa^+)^{V[G]}$.
The lifting criterion applied to $M\hookrightarrow T$ gives
$M[G]\prec T[G]$.
If $A\in\Pow(\kappa)\cap M[G]$ is nonstationary, elementarity gives
a club $C\in M[G]$ disjoint from $A$. Then,
$A\subseteq\kappa\setminus C\in I$.
\end{proof}

\begin{lemma}\label{lem:seed-patterns}
Let $\lambda<\kappa$ be infinite, let $W\prec H(\kappa^+)$ be a
weak $\kappa$-model with ${}^{\lambda}W\subseteq W$, and let
$c:[\kappa]^2\to\lambda$ be in $W$.
Suppose $\QQ$ is ${<}\lambda^+$-closed and forces that there is
an elementary embedding from $W$ to a transitive model with
critical point $\kappa$. Then, there is a $\QQ$-name $\dot f$
satisfying \eqref{eq:pattern-hypothesis}.
\end{lemma}
\begin{proof}
Work in a $\QQ$-generic extension and let $j:W\to N$ be such an
embedding. Define $f:\kappa\to\lambda$ by
\[
 f(\alpha)=j(c)(\alpha,\kappa).
\]
Since $\QQ$ adds no sequences of ordinals of length at most
$\lambda$, every $t\in[f]^{\leq\lambda}$ is in the ground model.
Hence, each such set is in $W$ by the closure assumption. An enumeration of $t$ of length
at most $\lambda$ shows that $j(t)=t$.
Thus, $D=R_c(t)\in W$ and $\kappa\in j(D)$ by elementarity of $j$.
If $D$ were nonstationary in the ground model, $W\prec H(\kappa^+)$ would
give a club $C\in W$ disjoint from $D$. Since $j$ fixes the
ordinals below $\kappa$, closure of $j(C)$ gives $\kappa\in j(C)$,
a contradiction. The maximum principle gives the required
name $\dot f$.
\end{proof}

\begin{proof}[Proof of Theorem~\ref{thm:main}]
Let $G\subseteq\PP=\Coll(\nu,{<}\kappa)$ be generic and work in
$\mathbb V=V[G]$. Fix $\omega\leq\lambda<\nu$ and
$c:[\kappa]^2\to\lambda$. By Lemma~\ref{lem:small-names}(1), with
$S=[\kappa]^2\times\lambda$, choose a name $\dot c\in H(\kappa^+)^V$ for $c$.
Choose $M\prec H(\kappa^+)$ containing this name, with
${}^{<\kappa}M\subseteq M$ and carrying an $M$-ultrafilter,
by Definition~\ref{def:height-ramsey}(1).
Let $j:M\to N$ be its well-founded ultrapower.

Lemma~\ref{lem:factor-lift} gives the tail $\QQ$ and lifts
$j_H:M[G]\to N[K]$.
The closure argument in Section~\ref{sec:collapse-setup} shows that
$\QQ$ is externally ${<}\nu$-closed and
${}^{\lambda}(M[G])\subseteq M[G]$ in $\mathbb V$.
By Lemma~\ref{lem:ambient-stationarity},
$W=M[G]\prec H(\kappa^+)^{\mathbb V}$.
Lemma~\ref{lem:seed-patterns} gives a $\QQ$-name $\dot f$
satisfying \eqref{eq:pattern-hypothesis}.

In $\mathbb V$, $\kappa=\nu^+$ and $\nu^{<\nu}=\nu$.
Indeed, the collapse adds no sequences of length less than $\nu$.
For each $\eta<\nu$, the ground cardinal $\nu^\eta$ is below
$\kappa$ and is collapsed to at most $\nu$.
Thus, $\theta^{<\lambda}\leq\nu^{<\nu}=\nu<\kappa$
for every infinite cardinal $\theta<\kappa$.
Lemma~\ref{lem:transfer} gives \eqref{eq:stationary-hubs}.
\end{proof}

\Needspace{9\baselineskip}
\section{Common-neighbor extraction}\label{sec:extraction}

Positive sets are ordered by inclusion unless a quotient is specified.

\begin{theorem}\label{thm:ideal-criterion}
Suppose $I$ is a proper uniform $\kappa$-complete ideal on $\kappa$
and $(I^+,\subseteq)$ has a countably closed dense subset.
For every $c:[\kappa]^2\to\omega$, there are $i<\omega$ and
$X\subseteq\kappa$ satisfying \eqref{eq:fcn}.
If $I$ is normal, they can be chosen with $\Gamma_i^X(F)\in I^+$
for every finite $F\subseteq X$.
\end{theorem}

\subsection{Frequent sequences and finite extensions}

Fix an ideal $I$ on $\kappa$ and a coloring $c:[\kappa]^2\to\omega$.
For $i<\omega$ and $B_0,B_1\in I^+$, call $(B_0,B_1)$
\emph{$i$-frequent}, and write $\Freq_i(B_0,B_1)$, if
\begin{equation}\label{eq:frequent}
 \forall D\in I^+\quad
 \left(D\subseteq B_1\ \Longrightarrow\
 \{z\in B_0:N_i(z)\cap D\in I\}\in I\right).
\end{equation}
This is the null-set form of \cite[Remark~5.2]{HSZ}.
For each positive $D\subseteq B_1$, we call
$\{z\in B_0:N_i(z)\cap D\in I\}$ the \emph{exceptional set}.
Frequency is preserved by positive restrictions in either coordinate.
The following construction is from \cite[Claim~5.3 and the proof of Theorem~5.1]{HSZ}.

\begin{lemma}\label{lem:sequence}
Suppose $I$ is proper, uniform and countably complete, and
$(I^+,\subseteq)$ has a countably closed dense subset.
For every $c:[\kappa]^2\to\omega$, there are $i<\omega$ and
$B_n\in I^+$, $n<\omega$, such that
\begin{equation}\label{eq:sequence}
 \forall n<m<\omega\quad \Freq_i(B_n,B_m).
\end{equation}
\end{lemma}
\begin{proof}
Fix a countably closed dense family $\mcD\subseteq I^+$.
We first show that every $A\in I^+$ contains a frequent pair in
some color. Suppose that $A$ contains no such pair, and choose
$C_0,D_0\in\mcD$ below $A$.

Given $C_n,D_n$, failure of $\Freq_n(C_n,D_n)$ gives
$D'\in I^+$ with $D'\subseteq D_n$ such that
$E=\{x\in C_n:N_n(x)\cap D'\in I\}\notin I$.
Choose $C_{n+1},D_{n+1}\in\mcD$ below $E,D'$, respectively.
Then, $N_n(x)\cap D_{n+1}\in I$ for every $x\in C_{n+1}$.
After the recursion, choose lower bounds $C_\omega,D_\omega\in\mcD$.
For $x\in C_\omega$ and every $n<\omega$, we have
$x\in C_{n+1}$ and $D_\omega\subseteq D_{n+1}$. Hence,
\[
 D_\omega=(D_\omega\cap\{x\})\cup
       \bigcup_{n<\omega}(D_\omega\cap N_n(x))\in I,
\]
a contradiction.

Next we show that there are $i<\omega$ and $A\in I^+$ such that
every positive subset of $A$ contains an $i$-frequent pair.
Otherwise, choose $A_0\in\mcD$. Given $A_n$, choose a positive
subset of $A_n$ containing no $n$-frequent pair, and refine it to
$A_{n+1}\in\mcD$. Then, a lower bound $A_\omega\in\mcD$ contains
no frequent pair in any color, contradicting the first part.

Fix such $A,i$ and let $R_0=A$. Given $R_n\in I^+$ with
$R_n\subseteq A$, choose $B_n,R_{n+1}\in I^+$ inside $R_n$ with
$\Freq_i(B_n,R_{n+1})$. For $n<m$, we have
$B_m\subseteq R_m\subseteq R_{n+1}$, so restriction gives
$\Freq_i(B_n,B_m)$.
\end{proof}

Fix a countably complete $I$ and a sequence satisfying
\eqref{eq:sequence} in color $i$. Let $B=\bigcup_{n<\omega}B_n$.
For finite $F\subseteq\kappa$, let
$D(F,m)=\Gamma_i^{B_m}(F)$ for $m<\omega$.
Call $F$ \emph{good} if there is $m_F<\omega$ such that
$D(F,m)\in I^+$ for all $m\geq m_F$.
The empty set is good. For good $F$, define
$E_F=\{z\in B:F\cup\{z\}\text{ is not good}\}$.
For nongood $F$, let $E_F=\varnothing$.

\begin{lemma}\label{lem:finite-extension}
For every finite $F\subseteq\kappa$, $E_F\in I$.
\end{lemma}
\begin{proof}
We may assume $F$ is good. Choose $m_F$ such that
$D(F,m)\in I^+$ for all $m\geq m_F$.
Fix $n<\omega$. For $m\geq\max\{m_F,n+1\}$, frequency gives
$\{z\in B_n:N_i(z)\cap D(F,m)\in I\}\in I$.
Outside the union of these countably many exceptional sets, every
$z\in B_n$ satisfies
\[
 D(F\cup\{z\},m)=D(F,m)\cap N_i(z)\in I^+
 \qquad(m\geq\max\{m_F,n+1\}).
\]
Hence, $F\cup\{z\}$ is good and $E_F\cap B_n\in I$.
Countable completeness gives $E_F=\bigcup_{n<\omega}(E_F\cap B_n)\in I$.
\end{proof}

\subsection{Extraction from ideals}

\begin{theorem}\label{thm:normal-extraction}
Suppose $I$ is a proper uniform normal $\kappa$-complete ideal and
$\langle B_n:n<\omega\rangle$ satisfies \eqref{eq:sequence}.
There is $X\subseteq B=\bigcup_{n<\omega}B_n$ such that
\[
 B\setminus X\in I,\qquad
 \forall F\in[X]^{<\omega}\quad\Gamma_i^X(F)\in I^+.
\]
\end{theorem}
\begin{proof}
Fix a bijection $q:[\kappa]^{<\omega}\to\kappa$ and let
$C=\{\beta<\kappa:q``[\beta]^{<\omega}\subseteq\beta\}$.
This set is club.
Lemma~\ref{lem:finite-extension} and normality give
$Z=\nabla_{\xi<\kappa}E_{q^{-1}(\xi)}\in I$.
Set $X=(B\cap C)\setminus Z$. Because $NS_\kappa\subseteq I$,
\[
 B\setminus X=(B\setminus C)\cup(B\cap Z)
              \subseteq(\kappa\setminus C)\cup Z\in I.
\]
Since $B_0\subseteq B$ is positive, $X\in I^+$.

Fix $F=\{\beta_0<\cdots<\beta_{r-1}\}\subseteq X$, and let
$F_j=\{\beta_\ell:\ell<j\}$ for $j\leq r$.
The set $F_0=\varnothing$ is good. If $F_j$ is good and $j<r$,
then $F_j\subseteq\beta_j$ and $\beta_j\in C$, so
$q(F_j)<\beta_j$. Since $\beta_j\notin Z$, we have
$\beta_j\notin E_{F_j}$. Hence, $F_{j+1}$ is good. Induction gives
that $F=F_r$ is good.

For finite $F\subseteq X$, choose $m$ with $D(F,m)\in I^+$.
Since $D(F,m)\subseteq B$ and $B\setminus X\in I$, the set
$D(F,m)\cap X\subseteq\Gamma_i^X(F)$ is positive.
\end{proof}

\begin{proof}[Proof of Theorem~\ref{thm:ideal-criterion}]
By Lemma~\ref{lem:sequence}, choose $i<\omega$ and
$\langle B_n:n<\omega\rangle$ satisfying \eqref{eq:sequence}, and
let $B=\bigcup_{n<\omega}B_n$. If $I$ is normal,
Theorem~\ref{thm:normal-extraction} gives the conclusion.

For the general case, suppose $Y\in[B]^{<\kappa}$ and every finite
subset of $Y$ is good. Given finite $F\subseteq Y$, choose $m$ with
$D(F,m)\in I^+$. By Lemma~\ref{lem:finite-extension}, uniformity
and $\kappa$-completeness,
\[
 E=Y\cup\bigcup_{H\in[Y]^{<\omega}}E_H\in I.
\]
Choose $x\in D(F,m)\setminus E$. Then, $x\notin Y$ is adjacent in
color $i$ to every member of $F$. Each finite subset of
$Y\cup\{x\}$ is good, since $x\notin E_H$ for every finite
$H\subseteq Y$.

Choose $h:\kappa\to[\kappa]^{<\omega}$ so that every fiber has
cardinality $\kappa$. For $a\in[\kappa]^{<\omega}$, let
$T_a=\{\alpha<\kappa:h(\alpha)=a\text{ and }a\subseteq\alpha\}$.
Only boundedly many points of the fiber fail $a\subseteq\alpha$,
so $|T_a|=\kappa$.
Recursively choose distinct $x_\alpha\in B$, applying the preceding
step with
\[
 Y=X_\alpha=\{x_\xi:\xi<\alpha\},\qquad
 F=F_\alpha=\{x_\xi:\xi\in h(\alpha)\cap\alpha\}.
\]
Every finite subset of $X_\alpha$ remains good. At a limit stage,
each finite subset is contained in an earlier stage.

Let $X=\{x_\alpha:\alpha<\kappa\}$. For finite $F\subseteq X$, the
index set $a=\{\xi<\kappa:x_\xi\in F\}$ is finite by distinctness.
For $\alpha\in T_a$, we have
\[
 F_\alpha=\{x_\xi:\xi\in h(\alpha)\cap\alpha\}
         =\{x_\xi:\xi\in a\}=F.
\]
Thus, $\{x_\alpha:\alpha\in T_a\}\subseteq\Gamma_i^X(F)$, so
$|\Gamma_i^X(F)|=\kappa$ by distinctness and $|T_a|=\kappa$.
\end{proof}

By Proposition~\ref{prop:positive-ideal-connectivity}, the normal
case also gives $\kappa\to_{I\text{-hc},<3}(\kappa)^2_\omega$.

Suppose, more generally, that $\mu<\kappa$ is uncountable
regular and the dense family in Theorem~\ref{thm:ideal-criterion}
is ${<}\mu$-closed. The same argument works with $\lambda$ colors for every infinite $\lambda<\mu$. Run the first two recursions
in Lemma~\ref{lem:sequence} through $\lambda$ colors, taking
lower bounds at nonzero limit stages $\delta\leq\lambda$.
The frequent sequence still has length $\omega$, and the
extraction arguments are unchanged. Thus, when $I$ is normal,
$\SFCN_\lambda(\kappa)$ holds.

\Needspace{8\baselineskip}
\section{Local ideals and stationary common neighbors}\label{sec:local-models}

\subsection{Preparatory facts}

\begin{lemma}\label{lem:closed-tail}
Let $\mu$ be regular uncountable.
Let $\calA\subseteq\Pow(\lambda)$ be a Boolean algebra closed under
unions of fewer than $\mu$ members, and let $I$ be a $\mu$-complete
ideal on $\calA$.
Suppose $\Phi:\calA/I\to\bbB$ is a Boolean embedding into a complete
Boolean algebra. Let $E\subseteq\bbB^+$ be ${<}\mu$-closed and dense,
with $E\subseteq\ran(\Phi)$. Then,
$\mcD=\{A\in\calA\setminus I:\Phi([A]_I)\in E\}$
is ${<}\mu$-closed and dense under actual inclusion.
\end{lemma}
\begin{proof}
For $A\in\calA\setminus I$, choose $e\in E$ below $\Phi([A]_I)$
and $B\in\calA$ with $\Phi([B]_I)=e$.
Then, $A\cap B\in\mcD$ is below $A$.

Let $\langle A_\xi:\xi<\delta\rangle$ be decreasing in $\mcD$,
where $0<\delta<\mu$. Choose $e\in E$ below all $\Phi([A_\xi]_I)$
and $B\in\calA$ with $\Phi([B]_I)=e$.
Let $B'=B\cap\bigcap_{\xi<\delta}A_\xi$.
Since $B\setminus A_\xi\in I$ for every $\xi<\delta$,
$\mu$-completeness gives $[B']_I=[B]_I$.
Thus, $B'\in\mcD$ is a lower bound.
\end{proof}

Let $M,N$ be transitive models of $\ZFC^-$ containing the same
infinite cardinal $\kappa$. If they have the same subsets of
$\kappa$, they have the same $H(\kappa^+)$
\cite[opening paragraph of Section~3]{Gitman}.
Here hereditary size is computed in each model.

\begin{lemma}\label{lem:name-agreement}
Suppose $M,N$ are transitive models of $\ZFC^-$ containing the
infinite cardinal $\kappa$, and
$\Pow(\kappa)\cap M=\Pow(\kappa)\cap N$.
Let $P$ be a forcing in $H(\kappa^+)$ in both models. If $G\subseteq P$
is generic over both models, then
\[
 \Pow(\kappa)\cap M[G]=\Pow(\kappa)\cap N[G].
\]
\end{lemma}
\begin{proof}
Let $A\subseteq\kappa$ be in $M[G]$.
The construction in Lemma~\ref{lem:small-names}(1), applied inside $M$
with $S=\kappa$, gives a $P$-name $\sigma\in H(\kappa^+)^M$
with $\sigma^G=A$.
Since $H(\kappa^+)^M=H(\kappa^+)^N$, we have $\sigma\in N$
and $A\in N[G]$.
The reverse inclusion follows in the same way.
\end{proof}

We use absoluteness of atomic forcing and its negations between
transitive models of $\ZFC^-$ containing the same poset and names.

\subsection{The seed ideal}\label{sec:seed-ideal}

An internal sequence is one in the given model. For an
ideal on $\Pow(\kappa)\cap W$, normality over $W$ means closure under
diagonal unions of sequences in $W$.

\begin{lemma}\label{lem:seed-ideal}
Let $M,N$ be transitive weak $\kappa$-models, and let $j:M\to N$
be the well-founded ultrapower by an $M$-ultrafilter $U$ on $\kappa$.
Let $P$ be a separative forcing in $H(\kappa^+)$ in both models.
Suppose that in $N$
there is an order isomorphism
\[
 \pi:j(P)\longrightarrow P\times Q,\qquad
 \pi(j(p))=(p,1_Q)\quad(p\in P),
\]
where $Q$ is separative. Let $G\subseteq P$ be generic over $V$,
and define $W=M[G]$, $\calA=\Pow(\kappa)\cap W$, and
$\bbB=\RO(Q)^{V[G]}$.
Let $\dot H$ be the canonical $Q$-name for the generic filter, and
let $\iota:Q\to\bbB^+$ be the canonical dense embedding, given by
$\iota(q)=\|\check q\in\dot H\|_{\bbB}$.
For $A=\tau^G\in\calA$, with $\tau\in M$, define
\[
 b_A=\bigl\|\check\kappa\in j(\tau)_G\bigr\|_{\bbB},
 \qquad I=\{A\in\calA:b_A=0\}.
\]
\begin{enumerate}
\item $I$ is proper and contains the bounded members of $\calA$ and
complements of clubs in $\kappa$ that are in $W$. It is closed under
unions of null sets indexed by internal sequences of length less than
$\kappa$, and under diagonal unions of internal $\kappa$-sequences
of null sets.
\item If $j$ is $\kappa$-powerset preserving, then for every
$\langle A_\xi:\xi<\kappa\rangle\in W$ from $\calA$, its null-index
set $\{\xi<\kappa:A_\xi\in I\}$ is in $W$.
\item Every $\iota(q)$, for $q\in Q$, equals $b_A$ for some
$A\in\calA$.
\end{enumerate}
\end{lemma}
\begin{proof}
For $H\subseteq Q$ generic over $V[G]$, let
$K=\pi^{-1}[G\times H]$. The product $G\times H$ is $V$-generic,
so $K$ is $N$-generic. Since $\pi(j(p))=(p,1_Q)$ for $p\in P$,
we have $j``G\subseteq K$. The lifting criterion from
Section~\ref{sec:collapse-setup} gives
\[
 j_H:M[G]\longrightarrow N[K],\qquad
 j_H(\tau^G)=j(\tau)^K=(j(\tau)_G)^H.
\]
If $\tau^G=\sigma^G$, every lift gives
$(j(\tau)_G)^H=(j(\sigma)_G)^H$.
Then, the forcing theorem gives
$\|\check\kappa\in j(\tau)_G\|_{\bbB}
 =\|\check\kappa\in j(\sigma)_G\|_{\bbB}$.
Thus, $b_A$ is well defined.
Every lift extends $j$, so $j_H(\kappa)=j(\kappa)>\kappa$.
Hence, $b_\kappa=1$. Elementarity gives
\[
 b_\varnothing=0,\qquad
 b_{\kappa\setminus A}=\neg b_A,\qquad
 b_{A\cap B}=b_A\wedge b_B.
\]
Thus, $I$ is proper. Bounded sets are null because the lifts fix
their bounds below $\kappa$.
Let $\vec A=\langle A_\xi:\xi<\delta\rangle\in W$ be a sequence from
$\calA$, where $\delta\leq\kappa$. Each lift fixes every index below
$\delta$. Hence,
\[
 \begin{aligned}
 b_{\bigcup_{\xi<\delta}A_\xi}
   &=\bigvee_{\xi<\delta}b_{A_\xi} &&(\delta<\kappa),\\
 b_{\nabla_{\xi<\kappa}A_\xi}
   &=\bigvee_{\xi<\kappa}b_{A_\xi} &&(\delta=\kappa).
 \end{aligned}
\]
Indeed, $j_H(\vec A)(\xi)=j_H(A_\xi)$ for $\xi<\delta$. When
$\delta=\kappa$, we have
\[
 \kappa\in j_H\Bigl(\nabla_{\xi<\kappa}A_\xi\Bigr)
 \quad\Longleftrightarrow\quad
 \exists\xi<\kappa\ \bigl(\kappa\in j_H(A_\xi)\bigr).
\]
This gives internal completeness and normality.
For a club $C\in W$, the set $j_H(C)$ is closed and contains $C$.
Hence, $\kappa\in j_H(C)$ and $b_{\kappa\setminus C}=0$.

For (2), assume $j$ is $\kappa$-powerset preserving. By
Lemma~\ref{lem:name-agreement},
\begin{equation}\label{eq:extension-agreement}
 \Pow(\kappa)\cap M[G]=\Pow(\kappa)\cap N[G].
\end{equation}
Fix $\vec A=\langle A_\xi:\xi<\kappa\rangle\in W$ and a name
$\sigma\in M$ for it. In $M$, define
\[
 \tau_\xi=\{(\check\alpha,r):\alpha<\kappa,\ r\in P,
       \ r\Vdash_P\check\alpha\in\sigma(\check\xi)\}.
\]
Then, $\tau_\xi^G=A_\xi$ and
$\vec\tau=\langle\tau_\xi:\xi<\kappa\rangle\in M$ by Replacement.
Hence,
\[
 j(\vec\tau)\res\kappa
       =\langle j(\tau_\xi):\xi<\kappa\rangle\in N,\qquad
 \langle j(\tau_\xi)_G:\xi<\kappa\rangle\in N[G].
\]
In $N[G]$, apply Separation to the definable atomic forcing relation
to form
\[
 E=\{\xi<\kappa:1_Q\Vdash_Q\check\kappa\notin j(\tau_\xi)_G\}.
\]
By atomic forcing absoluteness, $E$ is exactly
$\{\xi<\kappa:b_{A_\xi}=0\}$ in $V[G]$, and it is in $W$ by
\eqref{eq:extension-agreement}.

For (3), fix $q\in Q$ and let $r=\pi^{-1}((1_P,q))\in j(P)$.
By \eqref{eq:normal-form}, choose a function $h\in M$ with domain
$\kappa$ and $j(h)(\kappa)=r$. By \L{}o\'s's theorem,
$\{\alpha<\kappa:h(\alpha)\in P\}\in U$. Define $f_q(\alpha)=h(\alpha)$
on this set and $f_q(\alpha)=1_P$ elsewhere. Then, $f_q:\kappa\to P$
is in $M$ and $\pi(j(f_q)(\kappa))=(1_P,q)$. The name
$\dot X_q=\{(\check\alpha,f_q(\alpha)):\alpha<\kappa\}\in M$
has value $X_q=\{\alpha<\kappa:f_q(\alpha)\in G\}\in\calA$.
For every lift,
\[
 \kappa\in j_H(X_q)
 \quad\Longleftrightarrow\quad (1_P,q)\in G\times H
 \quad\Longleftrightarrow\quad q\in H.
\]
Hence, $b_{X_q}=\iota(q)$.
\end{proof}

\begin{lemma}\label{lem:closed-models}
In the setup of Lemma~\ref{lem:seed-ideal}, let $\mu\leq\kappa$
be regular uncountable. Suppose ${}^{<\mu}M\subseteq M$ in $V$.
Assume $P$ adds no sequences of length less than $\mu$ from
ground-model sets, and $N$ regards $Q$ as ${<}\mu$-closed.
In $V[G]$, we have ${}^{<\mu}W\subseteq W$ and $I$ is $\mu$-complete.
Moreover, $\mcD=\{A\in\calA\setminus I:b_A\in\iota[Q]\}$
is ${<}\mu$-closed and dense under actual inclusion.
\end{lemma}
\begin{proof}
By Section~\ref{sec:collapse-setup},
${}^{<\mu}W\subseteq W$ in $V[G]$, and $Q$ is externally
${<}\mu$-closed there. Hence, $\calA$ is closed under
unions of fewer than $\mu$ members, and the internal completeness
in Lemma~\ref{lem:seed-ideal}(1) is external.
The map $\Phi([A]_I)=b_A$ is a Boolean embedding from $\calA/I$
into $\RO(Q)^{V[G]}$.
Lemma~\ref{lem:seed-ideal}(3) gives $\iota[Q]\subseteq\ran(\Phi)$.
Apply Lemma~\ref{lem:closed-tail} with $E=\iota[Q]$.
\end{proof}

\subsection{Local extraction and forcing applications}

\begin{lemma}\label{lem:local}
Let $W$ be a transitive model of $\ZFC^-$ containing $\kappa$.
Assume $W^\omega\subseteq W$ in the ambient universe, and let
$\calA=\Pow(\kappa)\cap W$.
Suppose $I$ is an ideal on $\calA$ such that
\begin{enumerate}
\item $I$ has the properties stated in Lemma~\ref{lem:seed-ideal}(1).
\item The null-index conclusion of Lemma~\ref{lem:seed-ideal}(2)
holds for $I$ and $W$.
\item $(\calA\setminus I,\subseteq)$ has an externally countably
closed dense subset.
\end{enumerate}
For every $c:[\kappa]^2\to\omega$ in $W$, there are $i<\omega$ and
$X\in\calA\setminus I$ such that
$\Gamma_i^X(F)\in\calA\setminus I$ for every finite $F\subseteq X$.
\end{lemma}
\begin{proof}
Since $W^\omega\subseteq W$, the internal completeness in (1) gives
countable completeness of $I$.
Restrict the quantifier over $D$ in \eqref{eq:frequent} to
$\calA\setminus I$. The algebra $\calA$ is closed under countable
unions.
For $C,D\in\calA$ and $i<\omega$, condition (2) applied to
$\langle N_i(x)\cap D:x<\kappa\rangle$ gives
$\{x\in C:N_i(x)\cap D\in I\}\in W$.
These sets are in $\calA$, so the proof of Lemma~\ref{lem:sequence}
applies, using countable completeness of $I$ and the countably closed
dense family from (3). This gives $B_n\in\calA\setminus I$ with
$\Freq_i(B_n,B_m)$ for $n<m$. Since $W^\omega\subseteq W$,
$\langle B_n:n<\omega\rangle\in W$.

Let $B=\bigcup_nB_n$ and define
$D(F,m)=\Gamma_i^{B_m}(F)$.
Choose in $W$ a bijection
$h:\kappa\to[\kappa]^{<\omega}\times\omega$ and define
$R_\xi=D(F,m)$ when $h(\xi)=(F,m)$. The sequence
$\langle R_\xi:\xi<\kappa\rangle$ is in $W$, and so is
$T=\{\xi<\kappa:R_\xi\in I\}$ by (2). By Separation, the family
\[
 \mathcal H=\{F\in[\kappa]^{<\omega}:
       \exists m_0<\omega\ \forall m\geq m_0\quad
       h^{-1}((F,m))\notin T\}
\]
is in $W$. Since $R_{h^{-1}((F,m))}=D(F,m)$, these are exactly the good
finite sets.
For $F\in\mathcal H$, define $E_F=\{z\in B:F\cup\{z\}\notin\mathcal H\}$,
and let $E_F=\varnothing$ otherwise.
The sequence $\langle E_F:F\in[\kappa]^{<\omega}\rangle$ is in
$W$. The proof of Lemma~\ref{lem:finite-extension} gives $E_F\in I$:
its exceptional sets are in $\calA$ by (2), and their countable
unions are null by (1).

Choose in $W$ the coding $q$ and club $C$ from the proof of
Theorem~\ref{thm:normal-extraction}. The sequence
$\langle E_{q^{-1}(\xi)}:\xi<\kappa\rangle$ is in $W$, so (1)
gives $Z=\nabla_{\xi<\kappa}E_{q^{-1}(\xi)}\in I$.
Let $X=(B\cap C)\setminus Z$. Then, $X\in\calA$ and
$B\setminus X\subseteq(\kappa\setminus C)\cup Z\in I$ by (1).
The final argument of Theorem~\ref{thm:normal-extraction} now gives
$X\notin I$ and $\Gamma_i^X(F)\in\calA\setminus I$ for every finite
$F\subseteq X$.
\end{proof}

\begin{proof}[Proof of Theorem~\ref{thm:ramseycollapse}]
Let $G\subseteq P=\Coll(\nu,{<}\kappa)$ be generic and fix
$c:[\kappa]^2\to\omega$ in $V[G]$.
By Lemma~\ref{lem:small-names}(1), choose a name
$\dot c\in H(\kappa^+)^V$ for $c$.
Definition~\ref{def:height-ramsey}(2) gives a weak $\kappa$-model
$M\prec H(\kappa^+)$ containing $\dot c$, with $M^\omega\subseteq M$,
and a weakly amenable $M$-ultrafilter.
Its well-founded ultrapower $j:M\to N$ is $\kappa$-powerset preserving.

Apply Lemma~\ref{lem:seed-ideal} to the factorization in
Lemma~\ref{lem:factor-lift}.
Then, apply Lemma~\ref{lem:closed-models} with $\mu=\omega_1$. The resulting ideal on
$\Pow(\kappa)\cap M[G]$ satisfies the hypotheses of
Lemma~\ref{lem:local}. That lemma and
Lemma~\ref{lem:ambient-stationarity} give $\SFCN_\omega(\kappa)$
in $V[G]$, where $\kappa=\nu^+$.
\end{proof}

\begin{corollary}\label{cor:cohen}
Assume $\mathsf{CH}$ and let $\kappa$ be $T^{\kappa^+}_{\omega_1}$-Ramsey.
If $G\subseteq\operatorname{Add}(\omega_1,\kappa)$ is generic over
$V$, then
$V[G]\models\SFCN_\omega(\kappa)$.
In this extension, $\kappa=2^{\omega_1}$ is weakly inaccessible and
is not weakly compact.
\end{corollary}
\begin{proof}
Let $P=\operatorname{Add}(\omega_1,\kappa)$ consist of countable
partial functions from $\kappa\times\omega_1$ to $2$, ordered by
reverse inclusion. It is countably closed. The $\Delta$-system argument under $\mathsf{CH}$ gives
the $\omega_2$-c.c. Thus, $P$ preserves cardinals and cofinalities.

Since $\kappa$ is inaccessible, $\kappa^{\omega_1}=|P|=\kappa$
and $P\in H(\kappa^+)$.

The proof of Theorem~\ref{thm:ramseycollapse} gives
$\SFCN_\omega(\kappa)$ here as well. For the ultrapower
$j:M\to N$ in that proof, the coordinate split gives
$j(P)\cong P\times Q$ in $N$, where $Q$ is the tail on
$(j(\kappa)\setminus\kappa)\times\omega_1$.
The verification is the same as in the proof of
Lemma~\ref{lem:factor-lift}.

Counting names using the $\omega_2$-c.c. gives
$2^{\omega_1}\leq(|P|^{\omega_1})^V=\kappa$ in $V[G]$.
The $\kappa$ Cohen coordinates give equality.
Thus, $\kappa$ is weakly inaccessible but not weakly compact.
\end{proof}

\section{Partition hypotheses and coherent families}\label{sec:partition-hypotheses}

\subsection{Definitions}

Let $P$ be a nonempty directed quasi-order. For $r\geq1$, let
$P^{\leq r}=\bigcup_{1\leq j\leq r}P^j$. Tuples may have repeated
coordinates. Write $s\unlhd t$ when $s$ is obtained from $t$ by
deleting coordinates.

A \emph{full $r$-flag} is a tuple $(s_1,\ldots,s_r)$ with
$|s_j|=j$ and $s_1\unlhd\cdots\unlhd s_r$.
A map $F:P^{\leq r}\to P$ is \emph{$r$-cofinal} if
\[
 p\leq_P F((p))\quad(p\in P),\qquad
 s\unlhd t\ \Longrightarrow\ F(s)\leq_P F(t).
\]
The induced map on full flags is
$F^*((s_1,\ldots,s_r))=(F(s_1),\ldots,F(s_r))$.
For $n<\omega$ and a cardinal $\lambda$,
$\PH_n(P,\lambda)$ says that every coloring $c:P^{n+1}\to\lambda$
has an $(n+1)$-cofinal $F$ for which $c\circ F^*$ is constant on
all full $(n+1)$-flags. We use the general quasi-order formulation of
\cite[Sections~2 and~8, Definitions~3.1--3.2]{BBMT}.

For $\kappa=\omega_2$ and $\lambda=\omega$, the hypothesis below is
Komj\'ath--Shelah's principle $(*)$ \cite[p.~72]{KS}. We use the
common-neighbor construction from \cite[proof of Theorem~8.13]{BBMT}.

\begin{proposition}\label{prop:fcnph}
Let $\lambda$ be an infinite cardinal. Suppose that every coloring
$c:[\kappa]^2\to\lambda$ has an unbounded $X\subseteq\kappa$ and a
color $i<\lambda$ such that
\begin{equation}\label{eq:upper-neighbors}
 \alpha<\beta\text{ in }X
 \quad\Longrightarrow\quad
 \exists\gamma\in(\beta,\kappa)\quad
       c(\alpha,\gamma)=c(\beta,\gamma)=i.
\end{equation}
Then, $\PH_1(\kappa,\lambda)$ holds. In particular,
$\FCN_\lambda(\kappa)$ implies $\PH_1(\kappa,\lambda)$.
\end{proposition}
\begin{proof}
Given $d:\kappa^2\to\lambda$, set $c(\xi,\eta)=d(\xi,\eta)$ for
$\xi<\eta$.
Choose $X,i$ as in \eqref{eq:upper-neighbors} and enumerate $X$
increasingly as $\langle b_\alpha:\alpha<\kappa\rangle$.
Define $F((\alpha))=b_\alpha$. For $\alpha\ne\beta$, choose
$F((\alpha,\beta))$ above both $b_\alpha,b_\beta$ and joined to them
in color $i$, using \eqref{eq:upper-neighbors}. Define
$F((\alpha,\alpha))=F((\alpha,\alpha+1))$.

The map $F$ is $2$-cofinal, since $\alpha\leq b_\alpha$ and
$b_\alpha,b_\beta<F((\alpha,\beta))$. For every $\alpha,\beta<\kappa$,
\[
 d(F((\alpha)),F((\alpha,\beta)))
   =d(F((\beta)),F((\alpha,\beta)))=i.
\]
The FCN case follows because its common neighborhoods are unbounded.
\end{proof}

\begin{corollary}\label{cor:hubs-ph}
Let $\lambda<\kappa$ be infinite. If every
$c:[\kappa]^2\to\lambda$ has $B,i$ satisfying
\eqref{eq:stationary-hubs}, then $\PH_1(\kappa,\lambda)$ holds.
\end{corollary}
\begin{proof}
Given $c$, choose $B,i$ satisfying \eqref{eq:stationary-hubs}.
Use this property recursively to choose $b_\alpha\in B$ above all
earlier $b_\xi$, with $\Gamma_i^B(\{b_\xi,b_\alpha\})$ stationary
for every $\xi<\alpha$. The unbounded set
$X=\{b_\alpha:\alpha<\kappa\}$ satisfies \eqref{eq:upper-neighbors},
so Proposition~\ref{prop:fcnph} applies.
\end{proof}

Proposition~\ref{prop:fcnph} gives $\PH_1(\kappa)$ in the Cohen
model of Corollary~\ref{cor:cohen}.
Todorcevic--Zhang obtain all $\PH_n(\kappa)$ from a uniform normal
ideal whose quotient is a nontrivial measure algebra
\cite[Corollary~4.12, arXiv version~2]{TZ}.
They also obtain all $\PH_n(\kappa)$ after a length-$\kappa$
finite-support iteration of $\sigma$-centered forcing over a model
with a measurable $\kappa$
\cite[Corollary~4.13, arXiv version~2]{TZ}.
The iteration is c.c.c. and leaves $\kappa$ a regular limit cardinal.

\subsection{Tukey transfer}\label{sec:tukey}

Let $P$ and $Q$ be nonempty directed quasi-orders. We write
$P\leq_T Q$ if there is a map $g:P\to Q$ such that
$A_q=\{p\in P:g(p)\leq_Q q\}$
is bounded in $P$ for every $q\in Q$. Equivalently, $g$ sends
unbounded sets to unbounded sets. We call $g$ a \emph{Tukey map},
following \cite[Section~8]{BBMT}.

\begin{lemma}\label{thm:finite}
For every nonempty directed quasi-order $P$, every $n<\omega$, and
every $k<\omega$, $\PH_n(P,k)$ holds in $\ZFC$.
\end{lemma}
\begin{proof}
Let $m=n+1$ and choose an ultrafilter $U$ on $P$ containing all
upper cones.

Let $c:P^m\to k$ be given. Starting with $c_m=c$, define the functions
$c_r:P^r\to k$ for $r=m-1,\ldots,0$ by
\[
 c_r(a)=i\quad\Longleftrightarrow\quad
 \{x\in P:c_{r+1}(a\concat(x))=i\}\in U.
\]
Let $i=c_0(\varnothing)$. We construct $F:P^{\leq m}\to P$ by
tuple length, making it monotone under subtuples and requiring
\begin{equation}\label{eq:finite-invariant}
 c_r(F(s_1),\ldots,F(s_r))=i
 \qquad\text{for every full $r$-flag, }1\leq r\leq m.
\end{equation}
For each $p\in P$, the set
$\{x\in P:p\leq_P x\wedge c_1(x)=i\}$ is in $U$.
Choose $F((p))$ in this set.

Suppose $2\leq r\leq m$ and $F$ has been defined on all shorter
tuples with these properties. Fix $t\in P^r$. For every full $(r-1)$-flag
$(s_1,\ldots,s_{r-1})$ whose last tuple is a subtuple of $t$, define
\[
 D_{s_1,\ldots,s_{r-1}}
   =\{x\in P:c_r(F(s_1),\ldots,F(s_{r-1}),x)=i\}.
\]
By the definition of $c_{r-1}$, every such set is in $U$. Intersect these sets with
$\{x\in P:F(s)\leq_P x\}$ for each nonempty proper subtuple
$s\unlhd t$. All these sets are in $U$, so choose $F(t)$ in
their finite intersection.
\end{proof}

\begin{theorem}\label{thm:tukey}
Let $P\leq_T Q$ be nonempty directed quasi-orders. For every cardinal
$\lambda$ and every $n<\omega$,
$\PH_n(Q,\lambda)\ \Longrightarrow\ \PH_n(P,\lambda)$.
\end{theorem}
\begin{proof}
By Lemma~\ref{thm:finite}, we may assume $\lambda$ is infinite.

Fix a Tukey map $g:P\to Q$ and choose an upper bound $r(q)$ of
$A_q$ for each $q\in Q$. For $p\in P$ and $q\in Q$,
$g(p)\leq_Q q$ implies $p\leq_P r(q)$.
Define $b:Q^{n+1}\to2$ by
\[
 b(q_0,\ldots,q_n)=1
 \quad\Longleftrightarrow\quad
 \forall j<n\ \ g(r(q_j))\leq_Q q_{j+1}.
\]
Given $c:P^{n+1}\to\lambda$, define
\[
 d(q_0,\ldots,q_n)
 =\bigl(c(r(q_0),\ldots,r(q_n)),\ b(q_0,\ldots,q_n)\bigr).
\]
Identify $\lambda\times2$ with $\lambda$ and apply $\PH_n(Q,\lambda)$.
Choose an $(n+1)$-cofinal $F$ such that
$d\circ F^*$ is constantly $(i,\varepsilon)$.

We claim $\varepsilon=1$. Choose $t_1\in Q^1$.
For $1\leq j\leq n$, let $t_{j+1}=t_j\concat(g(r(F(t_j))))$.
Writing $q=g(r(F(t_j)))$, we have $(q)\unlhd t_{j+1}$, so
$g(r(F(t_j)))=q\leq_Q F((q))\leq_Q F(t_{j+1})$.
Thus, $(t_1,\ldots,t_{n+1})$ is a full flag with second color
coordinate $1$, so $\varepsilon=1$.

For $s\unlhd t$ with $|t|=|s|+1\leq n+1$, extend $s,t$ to a
full $(n+1)$-flag by deleting and appending coordinates.
Its second color coordinate gives
$g(r(F(s)))\leq_Q F(t)$.
Hence, $r(F(s))\leq_P r(F(t))$.
By transitivity, $r\circ F$ is monotone under all subtuple inclusions.

For $s\in P^{\leq n+1}$, apply $g$ coordinatewise and define
$H(s)=r(F(g\circ s))$.
For a singleton, $g(p)\leq_Q F((g(p)))$ gives
$p\leq_P H((p))$. Applying $g$ coordinatewise preserves tuple
lengths and subtuple inclusions. Thus, $H$ is $(n+1)$-cofinal. On any full
flag $(s_1,\ldots,s_{n+1})$, the first color coordinate gives
$c(H(s_1),\ldots,H(s_{n+1}))=i$.
\end{proof}

In particular, $\PH_n(-,\lambda)$ is invariant under Tukey equivalence
and passage to cofinal suborders.

\subsection{Coherent families and trivializations}

Let $A$ be an abelian group, and write $=^*$ for equality modulo
a finite set. We identify finite sets of ordinals with their
increasing enumerations. For $n\geq1$, an \emph{$n$-family} on
$\delta$ is a family
\[
 \Phi=\langle\phi_s:\min s\to A:s\in[\delta]^n\rangle.
\]
For $s=\{\alpha_0<\cdots<\alpha_n\}$, define
\[
 (d\Phi)_s
   =\sum_{j=0}^n(-1)^j\phi_{s\setminus\{\alpha_j\}}\res\alpha_0.
\]
A $0$-family is a single function $h:\delta\to A$, with
$(dh)_{\{\alpha\}}=h\res\alpha$.
We have $d^2=0$. The family $\Phi$ is \emph{$n$-coherent} if
$d\Phi=^*0$, and \emph{trivial} if $\Phi=^*d\Psi$ for some
$(n-1)$-family $\Psi$~\cite[Definition~2.25]{BLH}.
Since $d^2=0$, every trivial family is $n$-coherent.

In particular, one-coherence means
$h_\beta\res\alpha=^*h_\alpha$ for $\alpha<\beta<\delta$.
A two-coherent family satisfies
\[
 \phi_{\beta\gamma}\res\alpha-\phi_{\alpha\gamma}
       +\phi_{\alpha\beta}=^*0
 \qquad(\alpha<\beta<\gamma<\delta),
\]
and a trivializer satisfies
\begin{equation}\label{eq:two-trivialization}
 \psi_\beta\res\alpha-\psi_\alpha=^*\phi_{\alpha\beta}
 \qquad(\alpha<\beta<\delta).
\end{equation}
By \cite[Theorems~2.29--2.30]{BLH}, for $n\geq1$,
$\check H^n(\kappa,A_d)$ is isomorphic to the group of
$n$-coherent $A$-valued families on $\kappa$ modulo the trivial ones.

The next lemma uses compactness to handle uniformly bounded
coherence errors. Theorem~\ref{thm:ph-trivializes} combines it
with the formal-sum method of
\cite[proof of Theorem~7.5, especially Lemma~7.6]{BBMT}.

\begin{lemma}\label{lem:uniformly-bounded-coherence}
Let $n\geq1$ and let $\Phi$ be an $A$-valued $n$-family on
$\kappa$. If some $N<\omega$ satisfies
\[
 |\operatorname{supp}(d\Phi)_s|\leq N
 \qquad(s\in[\kappa]^{n+1}),
\]
then $\Phi$ is trivial.
\end{lemma}
\begin{proof}
We argue by induction on $n$. We first use first-order compactness
to obtain an abelian group $B\geq A$ and a $B$-valued
$(n-1)$-family $\Psi$ satisfying
\[
 |\operatorname{supp}(d\Psi-\Phi)_s|\leq N
 \qquad(s\in[\kappa]^n).
\]

Use the language of abelian groups $\{+,-,0\}$, expanded by
constants $c_a$ for $a\in A$. For $n\geq2$, also add a constant
$x_{t,\xi}$ for every $t\in[\kappa]^{n-1}$ and $\xi<\min t$.
It represents the desired value $\psi_t(\xi)$.
For $n=1$, use constants $x_{\varnothing,\xi}$ for $\xi<\kappa$,
representing the values of a single global function.
Let $T$ consist of the abelian-group axioms, the atomic diagram of $A$,
and the following sentences. For each
$s=\{\alpha_0<\cdots<\alpha_{n-1}\}\in[\kappa]^n$ and
$E\in[\alpha_0]^{N+1}$, include
\[
 \bigvee_{\xi\in E}
 \left(
   \sum_{j=0}^{n-1}(-1)^j
      x_{s\setminus\{\alpha_j\},\xi}
       =c_{\phi_s(\xi)}
 \right).
\]
These sentences allow at most $N$ errors on each $s$.

Fix a finite $T_0\subseteq T$.
Choose $\eta<\kappa$ greater than every ordinal in the finitely
many sets $s$ occurring in $T_0$.
Interpret the language in $A$, with $c_a=a$, and let
\[
 x_{t,\xi}=(-1)^{n-1}\phi_{t\cup\{\eta\}}(\xi)
\]
for every $x_{t,\xi}$ occurring in $T_0$.

For each $(s,E)$ occurring in $T_0$ and each $\xi\in E$,
\[
 \sum_{j=0}^{n-1}(-1)^j
      x_{s\setminus\{\alpha_j\},\xi}-\phi_s(\xi)
   =(-1)^{n-1}(d\Phi)_{s\cup\{\eta\}}(\xi).
\]
The right side is nonzero at at most $N$ coordinates.
Since $|E|=N+1$, the disjunction holds. Thus, $A$ satisfies $T_0$.

If $n=1$, compose the global function $\Psi:\kappa\to B$ with
any map $B\to A$ fixing $A$. Every equality
is preserved, so the resulting function trivializes $\Phi$.

For $n\geq2$, let $\overline\Psi$ be the image of $\Psi$ in $B/A$.
Then,
\[
 |\operatorname{supp}(d\overline\Psi)_s|\leq N.
\]
The induction hypothesis gives $\overline\Psi=^*d\overline\Theta$.
Lift $\overline\Theta$ to a $B$-valued family $\Theta$.
Every row of $\Psi-d\Theta$ is $A$-valued except at finitely many
coordinates. Replace those values by $0$ and call the resulting
family $\Psi'$. Since $d^2=0$,
$
 d\Psi'=^*d(\Psi-d\Theta)=d\Psi=^*\Phi
$.
\end{proof}

\begin{theorem}\label{thm:ph-trivializes}
Let $1\leq n<\omega$ and let $A$ be an abelian group.
If $\PH_n(\kappa)$ holds, every $n$-coherent $A$-valued family on
$\kappa$ is trivial. Thus, $\check H^n(\kappa,A_d)=0$.
\end{theorem}
\begin{proof}
Let $\Phi$ be such a family and let $\Theta=d\Phi$.
Every $\theta_s$ has finite support, and $d\Theta=d^2\Phi=0$.
Color each strictly increasing $(n+1)$-tuple $s$ by
$|\operatorname{supp}\theta_s|$, and all other tuples by $0$.
By $\PH_n(\kappa)$, choose an $(n+1)$-cofinal $F$ and $k<\omega$
such that this color is $k$ on every full-flag image.

We use finite integer sums of formal strictly increasing tuples
$[a]=[a_0,\ldots,a_{m-1}]$, including the empty tuple $[]$.
Tuples with repeated entries are interpreted as $0$. Define
\[
 \partial[a]
   =\sum_{j<m}(-1)^j[a_0,\ldots,\widehat a_j,\ldots,a_{m-1}],
 \qquad [a]*b=(-1)^m[a,b],
\]
where $b\geq\max a$. Extend both operations linearly, with
$\partial[]=0$ and $[]*b=[b]$. Alternating deletion gives
$\partial^2=0$ and
\begin{equation}\label{eq:tuple-extension}
 \partial(u*b)+ (\partial u)*b=u.
\end{equation}

Define linear maps $S,T$ on tuples of length at most $n+1$ by
$S([])=[]$, $T([])=0$,
\[
 S[s]=S(\partial[s])*F(s),\qquad
 T[s]=-\bigl([s]+T(\partial[s])\bigr)*F(s)
 \quad(s\ne\varnothing).
\]
The terms of $S[s]$ and $T[s]$ have lengths $|s|$ and $|s|+1$,
respectively. Their entries lie between $\min s$ and $F(s)$
by cofinality, so all appending operations are defined.
Induction gives
\begin{equation}\label{eq:tuple-identities}
 \partial S=S\partial,\qquad \partial T+T\partial=S-\mathrm{id}.
\end{equation}
Indeed, $\partial S(\partial[s])=S(\partial^2[s])=0$ gives the
first identity by \eqref{eq:tuple-extension}. For the second,
the inductive hypothesis gives
$\partial([s]+T(\partial[s]))=S(\partial[s])$.
Applying \eqref{eq:tuple-extension} to $T[s]$ gives
$\partial T[s]=S[s]-[s]-T(\partial[s])$.

For a finite sum $u=\sum_{r<\ell}m_r[t_r]$ of $(n+1)$-tuples,
define
\[
 \Theta(u)\res\alpha=\sum_{r<\ell}m_r(\theta_{t_r}\res\alpha)
 \qquad(\alpha\leq\min t_r\text{ for all }r<\ell).
\]
For a sum $v$ of $(n+2)$-tuples, the definition of $d\Theta$ gives
$\Theta(\partial v)=(d\Theta)(v)=0$ on any common domain.
Define
\[
 \psi_s=\phi_s+\Theta(T[s])\res\min s\qquad(s\in[\kappa]^n).
\]
Since $T[s]$ is a finite sum and every row of $\Theta$ has finite
support, $\Psi=^*\Phi$.
For $s=\{\alpha_0<\cdots<\alpha_n\}$, linearity and
\eqref{eq:tuple-identities} give
\[
 \begin{aligned}
 (d\Psi)_s
   &=\theta_s+\sum_{j=0}^{n}(-1)^j
       \Theta(T[s\setminus\{\alpha_j\}])\res\alpha_0\\
   &=\Theta([s]+T(\partial[s]))\res\alpha_0\\
   &=\Theta(S[s]-\partial T[s])\res\alpha_0\\
   &=\Theta(S[s])\res\alpha_0.
 \end{aligned}
\]
The last equality uses $\Theta(\partial T[s])=(d\Theta)(T[s])=0$.

For an increasing $m$-tuple $s$, expanding $S[s]$ gives at most
$m!$ signed terms of the form
\[
 [F(s_1),\ldots,F(s_m)],\qquad
 s_1\subsetneq\cdots\subsetneq s_m=s,\quad |s_j|=j.
\]
For $m=1$, $S[s]=[F(s)]$. For $m>1$, the $m$ subtuples
obtained by deleting one entry each contribute at most $(m-1)!$
terms. Appending $F(s)$ extends each flag by $s$.
For $m=n+1$, every nonzero term is a strictly increasing
full-flag image, so the corresponding row of $\Theta$ has
support of size $k$.
Hence,
\[
 |\operatorname{supp}(d\Psi)_s|\leq(n+1)!\,k.
\]
By Lemma~\ref{lem:uniformly-bounded-coherence}, choose $\Omega$
with $d\Omega=^*\Psi$. Then, $d\Omega=^*\Phi$.
\end{proof}

\Needspace{8\baselineskip}
\section{Forcing nontrivial coherent families}
\label{sec:ph-separation}

\subsection{Adding and threading a two-coherent family}\label{sec:adding-threading}

\begin{observation}\label{obs:last-row}
Let $\Phi$ be a two-coherent family on $\delta$ and let
$\Psi=\langle\psi_\alpha:\alpha\to A:\alpha<\delta\rangle$.
Then, $\Psi$ trivializes $\Phi$ if and only if adding
$\phi_{\alpha\delta}=-\psi_\alpha$ for $\alpha<\delta$ gives a
two-coherent family on $\delta+1$. In this case, adding
$\psi_\delta=0$ makes $\Psi$ a trivializer of the enlarged family.
If $\Psi$ trivializes $\Phi$, then $\Psi'$ trivializes $\Phi$ if and only if $\Psi'-\Psi$ is one-coherent.
\end{observation}
\begin{proof}
The new coherence equations are
\[
 -\psi_\beta\res\alpha+\psi_\alpha+\phi_{\alpha\beta}=^*0
 \qquad(\alpha<\beta<\delta),
\]
which are the trivializing equations for $\Psi$. The new
trivializing equations with $\psi_\delta=0$ hold exactly.
The last statement follows from
$d(\Psi'-\Psi)=^*d\Psi'-\Phi$.
\end{proof}

\begin{lemma}\label{lem:height-extension}
Let $\delta<\gamma$ be ordinals.
Every $A$-valued two-coherent family on $\delta+1$ extends to one
on $\gamma+1$.
\end{lemma}
\begin{proof}
By Observation~\ref{obs:last-row}, choose a trivializer
$\Psi$ of the given family $\Phi$.
Set $\psi_\alpha=0$ for $\delta<\alpha\leq\gamma$.
Keep all rows of $\Phi$ and define
\[
 \phi'_{\alpha\beta}=-\psi_\alpha
 \qquad(\alpha<\beta\leq\gamma,\ \delta<\beta).
\]
Then, $\Phi'=^*d\Psi$, with equality on the new rows, so
$d\Phi'=^*d^2\Psi=0$.
\end{proof}

Fix a nontrivial countable abelian group $A$ whose underlying set
is an ordinal at most $\omega$.
Work in a model $V$ with a weakly compact cardinal $\kappa$.
Let $P=\Coll(\omega_2,{<}\kappa)$.
For $G\subseteq P$ generic, define $\mathbb V=V[G]$.
In $\mathbb V$, $\kappa=\omega_3$.
Counting names using the $\kappa$-c.c. gives $2^{\omega_2}=\kappa$.
Hence, $\kappa^{<\kappa}=\kappa$.

In $\mathbb V$, let $R=\mathbb P(2,\kappa,A)$ be the forcing
of \cite[Definition~3.29]{BLH}. A condition is a two-coherent family
\[
 s=\langle\phi^s_{\alpha\beta}:\alpha\to A:
                    \alpha<\beta\leq\delta_s\rangle,
 \qquad \delta_s<\kappa,
\]
ordered by end extension. Call $\delta_s$ the height of $s$.
Since $\kappa^{<\kappa}=\kappa$, we have $|R|=\kappa$.
This forcing is ${<}\kappa$-distributive by
\cite[Lemma~3.30]{BLH}.
Its generic union is nontrivial
\cite[Lemma~3.31 and Remark~3.33]{BLH}.
We also use \cite[Claim~2.33]{BLH}.
Every one-coherent family on an ordinal of cofinality $\omega$ is trivial.
Every two-coherent family on a limit ordinal of cofinality less than
$\omega_2$ is trivial.

Let $r\subseteq R$ be generic over $\mathbb V$.
Define $\Phi=\bigcup r$ and $U=\mathbb V[r]$.
The forcing $R$ preserves cardinals and $\kappa^{<\kappa}=\kappa$.

In $\mathbb V$, the forcing $R$ is ${<}\omega_2$-closed. Let
$\langle s_\xi:\xi<\theta\rangle$ be decreasing in $R$, where
$\theta<\omega_2$. We may assume that
$\gamma=\sup_{\xi<\theta}\delta_{s_\xi}$ is not attained.
The union is two-coherent below $\gamma$, where
$\operatorname{cf}(\gamma)\leq\omega_1$.
Trivialize this family and apply Observation~\ref{obs:last-row}
to add its negative as the row at $\gamma$.

We will lift embeddings through $R$ by adding a trivializer of $\Phi$.
We use \cite[Definition~3.35]{BLH} with conditions of successor length.
A condition of $T(\Phi)$ is a family
\[
 t=\langle\psi^t_\alpha:\alpha\to A:
                 \alpha\leq\delta_t\rangle,
 \qquad \delta_t<\kappa,
\]
satisfying \eqref{eq:two-trivialization} for
$\alpha<\beta\leq\delta_t$.
The order is end extension.

\begin{lemma}\label{lem:closed-threading}
In $U$, the forcing $T(\Phi)$ is countably closed and separative.
For every $\eta<\kappa$, the set $\{t\in T(\Phi):\delta_t\geq\eta\}$
is dense in $T(\Phi)$.
In $\mathbb V$, the iteration
$R*T(\dot\Phi)$ has a dense separative ${<}\kappa$-closed suborder $D$
of size $\kappa$.
\end{lemma}
\begin{proof}
Let $t\in T(\Phi)$ with $\delta_t=\delta<\eta$.
By Observation~\ref{obs:last-row}, the family given by
$v_\alpha=-\phi_{\alpha\eta}$ for $\alpha<\eta$ and $v_\eta=0$
trivializes $\Phi$ through $\eta$. Hence,
$\langle\psi^t_\alpha-v_\alpha:\alpha\leq\delta\rangle$ is
one-coherent. Extend $\psi^t_\delta-v_\delta$ to a function $h:\eta\to A$.
Keep $\psi_\alpha=\psi^t_\alpha$ for $\alpha\leq\delta$, and define
$\psi_\alpha=h\res\alpha+v_\alpha$ for $\delta<\alpha\leq\eta$.
The difference $\langle\psi_\alpha-v_\alpha:\alpha\leq\eta\rangle$
is one-coherent. Observation~\ref{obs:last-row} gives the extension.

For countable closure, take a decreasing sequence
$\langle t_n:n<\omega\rangle$ and assume that
$\gamma=\sup_{n<\omega}\delta_{t_n}$ is not attained. Its union gives
$\langle\psi_\alpha:\alpha<\gamma\rangle$. By
Observation~\ref{obs:last-row}, the family
$\langle\psi_\alpha+\phi_{\alpha\gamma}:\alpha<\gamma\rangle$
is one-coherent. Since $\operatorname{cf}(\gamma)=\omega$, the vanishing result above
gives a trivializer $h$.
Add $\psi_\gamma=h$ to obtain a lower bound in $T(\Phi)$.

In $\mathbb V$, define
$D=\{(s,\check t):s\in R,\ \delta_s=\delta_t
                  \text{ and }t\text{ trivializes }s\}$,
which has size $\kappa$ by $\kappa^{<\kappa}=\kappa$.

To prove density, use the distributivity of $R$ to decide the
second coordinate as a family $t\in\mathbb V$.
By Lemma~\ref{lem:height-extension}, extend the first coordinate to
$s$ with $\eta=\delta_s\geq\delta_t$.
The first paragraph of this proof extends $t$ to a trivializer $t'$
of $s$ of height $\eta$. Then, $(s,\check t')\in D$.

Let $\langle(s_\xi,\check t_\xi):\xi<\theta\rangle$ be decreasing
in $D$, where $\theta<\kappa$, and assume that
$\gamma=\sup_{\xi<\theta}\delta_{s_\xi}$ is not attained.
The unions give a two-coherent family below $\gamma$ and a
trivializer $\langle\psi_\alpha:\alpha<\gamma\rangle$.
Extend both families by defining
\[
 \psi_\gamma=0,\qquad
 \phi_{\alpha\gamma}=-\psi_\alpha\quad(\alpha<\gamma).
\]
Observation~\ref{obs:last-row} shows that these families give a
lower bound in $D$.

In either $T(\Phi)$ or $D$, compatible conditions are comparable.
Given a proper extension of a condition, let $t$ be its threading
family and change $\psi^t_{\delta_t}(0)$, keeping all other data
fixed. Finite changes preserve the trivializing equations.
The result extends the original condition and is incompatible
with the given extension.
This proves separativity for both posets.
\end{proof}

The map $(s,\check t)\mapsto s$ from $D$ to $R$ is onto by
Observation~\ref{obs:last-row}.
It preserves order. Given $s'\leq s$, extend $t$ through
$\delta_{s'}$ to obtain an extension in $D$. Thus, the map is a projection.
Every condition of $T(\Phi)$ is in $\mathbb V$ by the distributivity
of $R$, and $\Phi\res(\delta+1)\in r$ for every $\delta<\kappa$.
Thus, the quotient over $r$ is isomorphic to $T(\Phi)$ via
\[
 t\longmapsto d_t=
       \bigl(\Phi\res(\delta_t+1),\check t\bigr).
\]
Indeed, if $(s,\check t)\in D$ and $s\in r$, then
$\delta_s=\delta_t$ gives $s=\Phi\res(\delta_t+1)$.

Let $C=\Coll(\omega_2,\kappa)^V$. Since $P$ adds no sequences of
length less than $\omega_2$, this is also the collapse computed in
$\mathbb V$. In $\mathbb V$, $D\times C$ is separative and ${<}\omega_2$-closed,
with size $\kappa=\omega_3$.
The $C$-coordinate collapses $\kappa$ to $\omega_2$.
Assigning distinct values at a new $C$-coordinate gives $\kappa$
pairwise incompatible extensions below every condition.
The collapse characterization
\cite[Theorem~14.1 and its proof]{Cummings} gives
\begin{equation}\label{eq:thread-absorption}
 \bigl(R*T(\dot\Phi)\bigr)\times C\simeq D\times C\simeq C.
\end{equation}
Here $\simeq$ denotes forcing equivalence.
The proof of the collapse characterization gives an order isomorphism
between dense suborders of $D\times C$ and $C$.
This isomorphism is in $H(\kappa^+)^{\mathbb V}$.

\subsection{Lifted embeddings and separation}

\begin{lemma}\label{lem:thread-lifts}
Work in $U$. For every $x\in H(\kappa^+)$, there are a weak
$\kappa$-model $W\prec H(\kappa^+)$ and a countably closed forcing $S$
such that $x\in W$ and $W^\omega\subseteq W$.
The forcing $S$ forces an elementary embedding from $W$ to a
transitive model with critical point $\kappa$.
\end{lemma}
\begin{proof}
Fix $x\in H(\kappa^+)^U$. By the maximum principle, choose a
$P$-name $\dot R$ whose value is $\mathbb P(2,\kappa,A)$ in every
$P$-generic extension. By the same principle, choose a sequence
$\langle\dot r_\xi:\xi<\kappa\rangle$ of $P$-names whose interpretations
enumerate this forcing in every such extension.
By Lemma~\ref{lem:small-names}(2), we may choose $\dot R$ and each
$\dot r_\xi$ in $H(\kappa^+)^V$, so the sequence is in $H(\kappa^+)^V$.
Let
\[
 \widehat P=\{(p,\dot r_\xi):p\in P,\ \xi<\kappa\},
\]
with the order from $P*\dot R$.
This suborder is dense and is in $H(\kappa^+)^V$.
For density, given $(p,\dot s)\in P*\dot R$, choose $q\leq_P p$ and
$\xi<\kappa$ with $q\Vdash_P\dot s=\dot r_\xi$.
Then, $(q,\dot r_\xi)\leq(p,\dot s)$.

By Lemma~\ref{lem:small-names}(2), choose a $\widehat P$-name
$\dot x\in H(\kappa^+)^V$ with $\dot x^{\widehat G}=x$,
where $\widehat G$ is the generic filter corresponding to $G*r$.
Choose in $V$ a weak $\kappa$-model $M\prec H(\kappa^+)$
with ${}^{<\kappa}M\subseteq M$ and an $M$-ultrafilter.
Require $A,P,\dot R,\dot x$, $\langle\dot r_\xi:\xi<\kappa\rangle$, and
$\widehat P$ to be in $M$.
Then, $R=\dot R^G\in M[G]$.
Let $j:M\to N$ be the well-founded ultrapower.
Since $A$ is hereditarily countable, $j(A)=A$.

Lemma~\ref{lem:ambient-stationarity} gives
$M[G]\prec H(\kappa^+)^{\mathbb V}$ and
$W=M[G,r]\prec H(\kappa^+)^U$.
We have $x\in W$ and $|W|=\kappa$.
Since $W=M[\widehat G]$ and $\widehat P$ adds no countable sequences,
Section~\ref{sec:collapse-setup} gives $W^\omega\subseteq W$.
The forcings $C,D$ are definable from $\kappa,A,R$ and are in
$H(\kappa^+)^{\mathbb V}$. By elementarity, $C,D\in M[G]$.
Choose in $M[G]$ an isomorphism $e:E\to F$ between dense suborders
$E\subseteq D\times C$ and $F\subseteq C$ witnessing
\eqref{eq:thread-absorption}.
For $a\in\Pow(\kappa)\cap M$, we have $a=j(a)\cap\kappa\in N$.
Every member of $M$ can be coded in $M$ by a well-founded relation
on an ordinal at most $\kappa$. These codes are in $N$, so $M\subseteq N$.
Hence, $e\in M[G]\subseteq N[G]$.

Use Lemma~\ref{lem:factor-lift} to identify $j(P)$ with $P\times Q$
via $\pi$. Separating the coordinate $\kappa$ identifies $Q$ with
$C\times Q'$. In $U$, let
\[
 S_0=T(\Phi)\times C\times Q'.
\]
Lemma~\ref{lem:closed-threading} gives countable closure of $T(\Phi)$.
The closure argument in Section~\ref{sec:collapse-setup} gives external
${<}\omega_2$-closure of $Q\cong C\times Q'$ in $\mathbb V$.
The distributivity of $R$ preserves this closure in $U$.
Thus, $S_0$ is countably closed.

Let $h=h_T\times h_C\times H'$ be $S_0$-generic over $U$, and define
\[
 h_D=\{(s,\check t)\in D:s\in r,\ t\in h_T\}.
\]
The projection $D\to R$ has quotient $T(\Phi)$.
Thus, $h_D\times h_C\times H'$ is
$D\times C\times Q'$-generic over $\mathbb V$.
Let $H_C$ be the filter on $C$ generated by
$e[(h_D\times h_C)\cap E]$, and let $H=H_C\times H'$.
Then, $H$ is $Q$-generic over $\mathbb V$.

\Needspace{9\baselineskip}
The filter $h_D\times h_C$ is generated by $e^{-1}[H_C\cap F]$.
The first projection of $h_D$ is $r$, and $
 h_T=\{t:\exists s\ ((s,\check t)\in h_D)\}
$. Since $e\in\mathbb V$, we obtain
$
 \mathbb V[H]
 =\mathbb V[h_D,h_C,H']
 =\mathbb V[r,h_T,h_C,H']
 =U[h].
$
\[
\begin{tikzcd}[column sep=3em, row sep=1.3em]
 \mathbb V \arrow[r, "{R,r}"] \arrow[dr, "{Q,H}"']
   & U=\mathbb V[r] \arrow[d, "{S_0,h}"] \\
   & {U[h]=\mathbb V[H].}
\end{tikzcd}
\]

Let $K=\pi^{-1}[G\times H]$. The lifting criterion gives
$
 j_H:M[G]\longrightarrow N[K]
$. The coordinate projection from $j(P)$ onto $C$ is in $N$
and maps $K$ onto $H_C$, so $H_C\in N[K]$.
Since $e\in N[G]\subseteq N[K]$, the same reconstruction works
inside $N[K]$. Thus,
$r,\Phi,\Psi\in N[K]$, where
$\Psi=\bigcup h_T=\langle\psi_\alpha:\alpha<\kappa\rangle$.
By Observation~\ref{obs:last-row}, extending $\Phi$ by
\[
 \phi_{\alpha\kappa}=-\psi_\alpha\qquad(\alpha<\kappa)
\]
gives a condition $m\in j_H(R)$ of height $\kappa$.
Every $s\in r$ has hereditary size less than $\kappa$ in $M[G]$,
so $j_H(s)=s$. Thus, $m\leq j_H(s)$ for every $s\in r$.

Since $P$ and $Q$ are countably closed, $U[h]=V[G,H]=V[K]$
adds no countable sequences over $V$.
Therefore, the closure criterion from Section~\ref{sec:collapse-setup}
gives $N[K]^\omega\subseteq N[K]$ in $U[h]$.
Thus, the internal countable closure of $j_H(R)$ in $N[K]$ is external.

Let $\dot R^*$ be an $S_0$-name for
$j_H(R)\mathbin{\downarrow}m$, and set $S=S_0*\dot R^*$.
Then, $S$ is countably closed.
If $g^*$ is $R^*$-generic over $U[h]$,
its upward closure $r^*$ in $j_H(R)$ is generic over $N[K]$.
Since every condition in $g^*$ is below $m$, we have $m\in r^*$.
Hence, $j_H``r\subseteq r^*$.
The lifting criterion now gives $j^*$ with critical point $\kappa$:
\[
\begin{tikzcd}[column sep=3em, row sep=1.2em]
 M \arrow[r, "j"] \arrow[d, "G"]
   & N \arrow[d, "K"] \\
 M[G] \arrow[r, "j_H"] \arrow[d, "r"]
   & N[K] \arrow[d, "r^*"] \\
 {W=M[G,r]} \arrow[r, "j^*"]
   & {N[K,r^*].}
\end{tikzcd}
\]
\end{proof}

\begin{lemma}\label{lem:coefficient-projections}
If $A=A_0\oplus A_1$ with both summands nontrivial, then both
coordinate projections of $\Phi$ are nontrivial in $U$.
\end{lemma}
\begin{proof}
Splitting every function into its two coordinates identifies $R$
with the equal-height dense suborder of
\[
 \mathbb P(2,\kappa,A_0)\times\mathbb P(2,\kappa,A_1).
\]
For either $\ell<2$, force with the other factor first.
Its $<\kappa$-distributivity leaves
$\mathbb P(2,\kappa,A_\ell)$ unchanged.
The remaining generic family is the $A_\ell$-projection of $\Phi$.
The nontriviality result recalled in Section~\ref{sec:adding-threading}
shows that it is nontrivial in $U$.
\end{proof}

\begin{proof}[Proof of Theorem~\ref{thm:ph-separation}]
Apply the construction with the countable group
\[
 A=\Z\oplus\bigoplus_{p\text{ prime}}\Z/p\Z,
\]
and work in the resulting model $U$.
By Lemma~\ref{lem:coefficient-projections}, the projected family
on each displayed cyclic summand $C$ is nontrivial.
Every nontrivial abelian group $B$ contains such a subgroup $C$.
The projected family remains nontrivial with coefficients in $B$.
Indeed, choose representatives $r_D$ for the cosets $D\in B/C$
and define $\rho(b)=b-r_{b+C}\in C$. Whenever $b-b'\in C$,
\[
 \rho(b)-\rho(b')=b-b'.
\]
Thus, composing a $B$-valued trivializer with $\rho$ would give
a $C$-valued trivializer, a contradiction.
Hence, $\check H^2(\omega_3,B_d)\ne0$ for every nontrivial
abelian group $B$.
Given $c:[\kappa]^2\to\omega$,
Lemma~\ref{lem:thread-lifts} gives $W,S$ for $x=c$.
Lemmas~\ref{lem:seed-patterns} and~\ref{lem:transfer}, with
$\lambda=\omega$, give \eqref{eq:stationary-hubs}.
\end{proof}

\section{Further consequences}\label{sec:consequences}

\subsection{Products of the threading forcing}
\label{sec:thread-products}

For regular $\kappa\geq\omega_2$, the forcings
$\mathbb P(n,\kappa,A)$ for $n=1,2$ are $\kappa$-strategically
closed~\cite[Lemma~3.30]{BLH}.
The generic two-coherent family is nontrivial
\cite[Lemma~3.31 and Remark~3.33]{BLH}.

\begin{proposition}\label{prop:thread-products}
Let $\kappa\geq\omega_2$ be regular and let $A$ be a nontrivial
abelian group. Let $R=\mathbb P(2,\kappa,A)$, and let $\dot\Phi$
name its generic two-coherent family.
Write $\dot T=T(\dot\Phi)$ for the threading forcing from
Section~\ref{sec:adding-threading}, and let
\[
 D=\{(s,\check t):s\in R,\ \delta_s=\delta_t,\
                         t\text{ trivializes }s\}.
\]
For every ordinal $0<\eta<\kappa$, with full-support products,
\begin{equation}\label{eq:thread-products}
 R*\dot T^\eta\simeq
 D\times\mathbb P(1,\kappa,A)^{\eta\setminus\{0\}}.
\end{equation}
In the $R$-extension, every $T^\eta$ is ${<}\kappa$-distributive
and preserves the regularity of $\kappa$.
The forcing $T$ trivializes the nontrivial family $\Phi$.
\end{proposition}
\begin{proof}
Restrict the left-hand side of \eqref{eq:thread-products} to
conditions $(s,\check u)$, where
$u=\langle t_\xi:\xi<\eta\rangle$ and each $t_\xi$ has height
$\delta_s$ and trivializes $s$.
To prove density, use the ${<}\kappa$-distributivity of $R$ to decide
all $t_\xi$ in $V$.
Regularity of $\kappa$ gives a bound below $\kappa$ for their heights.
By Lemma~\ref{lem:height-extension}, extend the first coordinate to
$s$ above this bound. Extend each $t_\xi$ to a trivializer of $s$
as in the first paragraph of the proof of Lemma~\ref{lem:closed-threading}.

The equal-height suborder of the right-hand side is also dense.
The $D$-coordinate extends to arbitrarily large heights.
A one-coherent condition extends to any larger height: extend its
last function and use its restrictions for the new rows.
Regularity of $\kappa$ gives a common bound for the $\eta$ heights.

Between these dense suborders, define
\[
 (s,\check u)\longmapsto
 \bigl((s,\check t_0),
       \langle t_\xi-t_0:0<\xi<\eta\rangle\bigr).
\]
Observation~\ref{obs:last-row} shows that each difference is
one-coherent. Conversely, adding a one-coherent family of the same
height to $t_0$ gives a trivializer of $s$.
Thus, the inverse adds the corresponding family to $t_0$ in each
coordinate. Both maps preserve end extension, so \eqref{eq:thread-products} holds.

The forcing $D$ is ${<}\kappa$-closed, and
$\mathbb P(1,\kappa,A)$ is $\kappa$-strategically closed.
Therefore, their full-support product is ${<}\kappa$-strategically
closed. Hence, it is ${<}\kappa$-distributive.
By \eqref{eq:thread-products}, $R*\dot T^\eta$ is ${<}\kappa$-distributive.
Thus, $R$ forces that $T^\eta$ is ${<}\kappa$-distributive.
\end{proof}

For $\eta=2$, the forcing $T\times T$ preserves
$\operatorname{cf}(\kappa)=\kappa\geq\omega_2$, but $T$ trivializes $\Phi$.
This answers \cite[Question~4.2]{BLH} negatively for two-coherent families.
No large-cardinal assumption is needed.

\subsection{Simultaneous successors}\label{sec:applications}

Start with a huge cardinal $\theta$.
Eskew--Hayut \cite[Theorem~56]{EH} give a forcing extension $W$
in which $\theta$ remains inaccessible.
In $V^W_\theta$, every regular $\mu$ has a proper uniform normal
$\mu^+$-complete ideal $I_\mu$ on $\mu^+$ with
$\Pow(\mu^+)/I_\mu\cong\RO(\Coll(\mu,\mu^+))$.

\begin{corollary}\label{cor:global-successors}
Relative to a huge cardinal, there is a model of $\ZFC$ such that,
for every regular $\mu\geq\omega_1$ and every infinite $\lambda<\mu$,
\[
 \SFCN_\lambda(\mu^+)
 \quad\text{and}\quad
 \mu^+\to_{NS_{\mu^+}\text{-hc},<3}(\mu^+)^2_\lambda.
\]
\end{corollary}
\begin{proof}
Work in $V^W_\theta$ and fix a regular $\mu\geq\omega_1$.
The collapse $\Coll(\mu,\mu^+)$ is ${<}\mu$-closed.
Apply Lemma~\ref{lem:closed-tail} to the quotient isomorphism above.
It gives a ${<}\mu$-closed dense family in $(I_\mu^+,\subseteq)$.
The $\lambda$-color version of Theorem~\ref{thm:ideal-criterion},
at the end of Section~\ref{sec:extraction}, gives $\SFCN_\lambda(\mu^+)$
for every infinite $\lambda<\mu$.
Proposition~\ref{prop:positive-ideal-connectivity} gives
$\mu^+\to_{NS_{\mu^+}\text{-hc},<3}(\mu^+)^2_\lambda$.
\end{proof}

Eskew--Hayut \cite[Theorem~75]{EH} give a model in which
$\PH_n(\omega_m)$ holds for all $n<m<\omega$.

\subsection{Products of directed orders}\label{sec:products}

\begin{lemma}\label{lem:ph-product}
Let $P,Q$ be nonempty directed quasi-orders, let $n<\omega$, and let
$\lambda$ be an infinite cardinal. Let $\tau=\lambda^{|P^{n+1}|}$.
If $\PH_n(P,\lambda)$ and $\PH_n(Q,\tau)$ hold, then
$\PH_n(P\times Q,\lambda)$ holds, where $P\times Q$ has the
product order.
\end{lemma}
\begin{proof}
Given $c:(P\times Q)^{n+1}\to\lambda$, define
\[
 d:Q^{n+1}\longrightarrow{}^{P^{n+1}}\lambda,
 \qquad
 d(q_0,\ldots,q_n)(p_0,\ldots,p_n)
   =c((p_0,q_0),\ldots,(p_n,q_n)).
\]
The range of $d$ has size at most $\tau$.
By $\PH_n(Q,\tau)$, choose an $(n+1)$-cofinal map $F_Q$ and
$h$ such that $d\circ F_Q^*$ is constantly $h$.
Apply $\PH_n(P,\lambda)$ to $h$.
Choose an $(n+1)$-cofinal map $F_P$ and a color $i$ such that
$h\circ F_P^*$ is constantly $i$.

For $s\in(P\times Q)^{\leq n+1}$, let $s_P,s_Q$ be its coordinate
projections and define
$F(s)=(F_P(s_P),F_Q(s_Q))$.
The projections preserve lengths and subtuple inclusions, so $F$
is $(n+1)$-cofinal. On every full $(n+1)$-flag $(s_1,\ldots,s_{n+1})$,
\[
 c(F(s_1),\ldots,F(s_{n+1}))
   =h(F_P((s_1)_P),\ldots,F_P((s_{n+1})_P))=i.\qedhere
\]
\end{proof}

\begin{corollary}\label{cor:product-consistency}
Suppose $\nu<\kappa<\mu<\Lambda$, where $\nu,\mu$ are uncountable
regular cardinals, $\kappa,\Lambda$ are weakly compact, and
$2^\kappa<\mu$. Let $G\subseteq\Coll(\nu,{<}\kappa)$ be generic
over $V$, and let $H\subseteq\Coll(\mu,{<}\Lambda)$ be generic
over $V[G]$. Then, $V[G][H]$ satisfies
$P\leq_T\nu^+\times\mu^+\Longrightarrow\PH_1(P)$
for every nonempty directed quasi-order $P$.
In particular, fix integers $2\leq a$ and $b\geq a+3$.
From two weakly compact cardinals, we obtain a model in which
$\PH_1(P)$ holds for every nonempty directed quasi-order
$P\leq_T\omega_a\times\omega_b$.
\end{corollary}
\begin{proof}
The first collapse has size $\kappa<\Lambda$.
Thus, $\Lambda$ remains weakly compact in $V[G]$.
Theorem~\ref{thm:main} and Corollary~\ref{cor:hubs-ph} give
$\kappa=\nu^+$ and $\PH_1(\kappa)$ there.
Lemma~\ref{lem:small-names}(1) gives at most $(2^\kappa)^V$
names for subsets of $\kappa$. The $\kappa$-c.c. preserves this
cardinal, so $2^\kappa=(2^\kappa)^V<\mu$ in $V[G]$.

The second collapse is ${<}\mu$-closed, so it adds no
$\kappa$-sequences of ordinals. Hence, it adds no colorings
$\kappa^2\to\omega$ and preserves $2^\kappa$.
The old $2$-cofinal witnesses still work, so $\PH_1(\kappa)$ is preserved.
Apply Theorem~\ref{thm:main} in $V[G]$, followed by
Corollary~\ref{cor:hubs-ph} with $\lambda=2^\kappa$.
This gives $\Lambda=\mu^+$ and $\PH_1(\Lambda,2^\kappa)$.
Since $\omega^{|\kappa^2|}=2^\kappa$, Lemma~\ref{lem:ph-product} gives
$\PH_1(\nu^+\times\mu^+)$. Theorem~\ref{thm:tukey} gives $\PH_1(P)$.

For the last statement, start with weakly compact $\kappa<\Lambda$.
If necessary, first force with $\Coll(\kappa^+,(2^\kappa)^V)$ to
get $2^\kappa=\kappa^+$. This forcing is ${<}\kappa^+$-closed, so it adds no subsets of $\kappa$.
In particular, it adds no two-colorings of $[\kappa]^2$.
The old homogeneous sets still witness weak compactness of $\kappa$.
Its size is below $\Lambda$, so $\Lambda$ remains weakly compact.
In this intermediate model, use
$\nu=\omega_{a-1}$ and $\mu=\kappa^{+(b-a-1)}$.
Since $b-a-1\geq2$, we have $2^\kappa=\kappa^+<\mu$.
The two collapses make $\kappa=\omega_a$ and $\Lambda=\omega_b$.
\end{proof}

The restriction $b\geq a+3$ comes from $2^\kappa<\mu$.
These products are not Tukey equivalent to ordinals.
Recall that the additivity of a directed order is the least size
of an unbounded subset. Both additivity
and cofinality are preserved by Tukey equivalence. For infinite
regular $\theta<\eta$,
$
 \operatorname{add}(\theta\times\eta)=\theta
 <\eta=\operatorname{cof}(\theta\times\eta).
$
Any Tukey-equivalent ordinal would have cofinality $\eta$.
Hence, it would have additivity $\eta$, a contradiction.

\begin{corollary}\label{cor:product-equiconsistency}
$\PH_1(\omega_2\times\omega_5)$ is equiconsistent
over $\ZFC$ with the existence of two weakly compact cardinals.
\end{corollary}
\begin{proof}
Corollary~\ref{cor:product-consistency}, with $a=2$ and $b=5$,
gives the upper bound. Conversely, both $\omega_2$ and $\omega_5$
are Tukey below $\omega_2\times\omega_5$, so
Theorem~\ref{thm:tukey} gives $\PH_1(\omega_2)$ and $\PH_1(\omega_5)$.
The lower bound in Section~\ref{sec:introduction} shows that
both ordinals are weakly compact in $L$.
\end{proof}

\subsection{Countable additivity}\label{sec:additivity}

Let $P_j$, $j<\omega$, be nonempty directed posets and let
$\mathbf G_j$ be $P_j$-indexed inverse systems of abelian groups,
with groups $G^j_p$ for $p\in P_j$. Give
$P=\prod_{j<\omega}P_j$ the product order. Their direct-sum
system $\mathbf G$ is indexed by $P$, with
$G_x=\bigoplus_{j<\omega}G^j_{x(j)}$ and coordinatewise bonding maps.
Write $\lim_P^1$ for the first derived inverse limit over $P$.
If $\PH_1(P)$ holds, the canonical map
\begin{equation}\label{eq:lim1-additivity}
 \bigoplus_{j<\omega}\lim\nolimits_{P_j}^{1}\mathbf G_j
 \longrightarrow\lim\nolimits_P^{1}\mathbf G
\end{equation}
is an isomorphism by \cite[Theorem~8.4]{BBMT}, with
$n=1$, $I=\omega$, and $S=\varnothing$.

\begin{corollary}\label{cor:lim1-additivity}
In the model of Corollary~\ref{cor:product-consistency}, the map
\eqref{eq:lim1-additivity} is an isomorphism for every such family
with $P_j\leq_T\nu^+\times\mu^+$ for all $j<\omega$.
\end{corollary}
\begin{proof}
Choose Tukey maps $g_j=(g_j^0,g_j^1):P_j\to\nu^+\times\mu^+$ and define
\[
 g:P\longrightarrow\nu^+\times\mu^+,\qquad
 g(x)=\left(\sup_{j<\omega}g_j^0(x(j)),\,
            \sup_{j<\omega}g_j^1(x(j))\right).
\]
Regularity gives $g(x)\in\nu^+\times\mu^+$. Given $(\alpha,\beta)\in\nu^+\times\mu^+$, choose
$p_j\in P_j$ bounding
$\{p\in P_j:g_j(p)\leq(\alpha,\beta)\}$ for each $j<\omega$.
If $g(x)\leq(\alpha,\beta)$, then $x(j)\leq_{P_j}p_j$ for every $j$.
Thus, $\{x\in P:g(x)\leq(\alpha,\beta)\}$ is bounded by
$\langle p_j:j<\omega\rangle$, so $g$ is Tukey.
Corollary~\ref{cor:product-consistency} gives $\PH_1(P)$.
Thus, \eqref{eq:lim1-additivity} is an isomorphism.
\end{proof}

Fix integers $2\leq a$ and $b\geq a+3$.
From two weakly compact cardinals, we obtain a model in which
\eqref{eq:lim1-additivity} is an isomorphism for every countable
mixture of $\omega_a$-indexed, $\omega_b$-indexed, and
$\omega_a\times\omega_b$-indexed inverse systems.

\Needspace{8\baselineskip}
\section{Questions}\label{sec:questions}

\begin{question}\label{ques:consistency-strength}
What are the exact consistency strengths of
$\omega_2\to_{\hc,<4}(\omega_2)^2_\omega$,
$\omega_2\hcarr(\omega_2)^2_\omega$, and $\SFCN_\omega(\omega_2)$?
\end{question}

\begin{question}\label{ques:finite-bound-separation}
Is it consistent that a regular $\kappa>\omega_1$ satisfies
$\kappa\to_{\hc}(\kappa)^2_\omega$ but
$\kappa\nrightarrow_{\hc,<m}(\kappa)^2_\omega$ for every integer $m\geq3$?
Can this hold for $\kappa=\omega_2$?
\end{question}

\begin{question}\label{ques:fcn-separation}
Is it consistent that a regular uncountable $\kappa$ satisfies
$\kappa\hcarr(\kappa)^2_\omega$ and $\neg\FCN_\omega(\kappa)$?
\end{question}

\begin{question}\label{ques:ph-hc-separation}
Is it consistent with $\mathsf{CH}$ that $\PH_1(\omega_2)$ holds and
$\omega_2\nrightarrow_{\hc}(\omega_2)^2_\omega$?
\end{question}

Without $\mathsf{CH}$, the Mitchell model obtained from a measurable
cardinal gives this separation. Its uniform normal $2$-precipitous
ideal gives $\PH_1(\omega_2)$ by a recursion using
\cite[Claims~3.4--3.5 and proof of Corollary~4.9]{HSZ} and
Proposition~\ref{prop:fcnph}. Since $\mathfrak c=\omega_2$,
the negative highly connected relation follows from $\ZFC$.

For regular uncountable $\kappa$ (see \cite{Bergfalk,KS} for notation),
\[
 \PH_1(\kappa)\ \Longrightarrow\
 \kappa\to_{\mathrm{wc}}(\kappa)^2_\omega
 \ \Longleftrightarrow\
 \kappa\to(\operatorname{Top}K_\kappa)^2_\omega.
\]
The first implication is in \cite[Section~5, final paragraph]{LH}.
Bergfalk \cite[Theorem~0.11 and its proof, arXiv version~1]{Bergfalk}
showed that the Mitchell model obtained from a weakly compact
cardinal satisfies
\[
 \omega_2\to_{\mathrm{wc}}(\omega_2)^2_\omega
 \quad\text{and}\quad
 \omega_2\nrightarrow_{\mathrm{hc}}(\omega_2)^2_\omega.
\]

\begin{question}\label{ques:remaining-targets}
Is $\aleph_{\omega+1}\hcarr(\aleph_{\omega+1})^2_\omega$ consistent,
and can it be strengthened to $\SFCN_\omega(\aleph_{\omega+1})$?
Can $\kappa\hcarr(\kappa)^2_\omega$ hold when $\kappa$ is the least
weakly inaccessible cardinal, or the least strongly inaccessible
cardinal?
\end{question}

With $<4$ in place of $<3$, these instances have positive answers
\cite[Theorem~5.1 and Section~6]{HSZ}.
For the least-inaccessible instances, start with a strongly inaccessible
cardinal above the supercompact. Then, the target cardinals
exist after the collapse.

\section*{Acknowledgments}
I am grateful to Professor Itay Neeman for bringing several
of the cited papers to my attention and for discussing the
proofs with me.

I used ChatGPT for grammar correction and assistance with
proof checking. During proof checking, it suggested a
simplification of my original proof of the general case of
\Cref{thm:ideal-criterion}, in which normality is not assumed.
I verified the simplified proof
and included it in this version. Apart from this simplification, all proofs
are my own.

\enlargethispage{3\baselineskip}
\end{document}